\documentclass[12pt,a4paper,reqno]{amsart}
\usepackage{amsmath}
\usepackage{amsfonts}
\usepackage{amssymb}
\usepackage{amsthm}
\usepackage{enumerate}
\usepackage{centernot}
\usepackage{fullpage}
\usepackage{mathrsfs}
\usepackage{graphicx}
\usepackage{color}

\renewcommand{\bf}{\mathbf}

\newcommand{\upperRomannumeral}[1]{\uppercase\expandafter{\romannumeral#1}}

\theoremstyle{plain}
  
  \newtheorem{claim}[]{Claim}
  \newtheorem{proposition}[]{Proposition}
  \newtheorem{lemma}[]{Lemma}
  \newtheorem{theorem}[]{Theorem}
  \newtheorem{corollary}[]{Corollary}
  
  \newtheorem{remark}[]{Remark}

\title[Ursell functions]{Monotonicity of Ursell functions in the Ising model}

\author{Federico Camia}
\address{Division of Science, NYU Abu Dhabi, Saadiyat Island, Abu Dhabi, UAE \& Courant Institute of Mathematical Sciences, New York University, 251 Mercer st, New York, NY 10012, USA.}
\email{federico.camia@nyu.edu}
\author{Jianping Jiang}
\address{Yau Mathematical Sciences Center, Tsinghua University, Beijing 100084, China.}
\email{jianpingjiang@tsinghua.edu.cn}
\author{Charles M. Newman}
\address{Courant Institute of Mathematical Sciences, New York University, 251 Mercer Street, New York, NY 10012, USA \& NYU-ECNU Institute of Mathematical Sciences at NYU Shanghai, 3663 Zhongshan Road North, Shanghai 200062, China.}
\email{newman@cims.nyu.edu}

\begin{document}
\begin{abstract}
	In this paper, we consider Ising models with ferromagnetic pair interactions. 
We prove that the Ursell functions $u_{2k}$  satisfy: $(-1)^{k-1}u_{2k}$ is increasing 
in each interaction. As an application, we prove a 1983 conjecture by Nishimori and Griffiths 
about the partition function of the Ising model with complex external field $h$: its closest zero to the origin (in the variable $h$) moves towards the origin as an arbitrary interaction increases.
\end{abstract}

\maketitle

\section{Introduction}
\subsection{Monotonicity of Ursell functions}
For a family of random variables $\{\sigma_1,\dots,\sigma_k\}$, its \textit{Ursell function} $u_k(\sigma_1,\dots,\sigma_k)$ may be defined by
\begin{equation}\label{eq:ursdef1}
	u_k(\sigma_1,\dots,\sigma_k):=\left.\frac{\partial^k}{\partial h_1\dots\partial h_k}\ln\left\langle \exp\left[\sum_{j=1}^k h_j\sigma_j\right]\right\rangle\right|_{\bf{h}=\bf{0}},
\end{equation}
where $\langle \cdot\rangle$ denotes the expectation and $\bf{h}=(h_1,\dots,h_k)$. Equivalently, it may be defined as
\begin{equation}\label{eq:ursdef2}
	u_k(\sigma_1,\dots,\sigma_k)=\sum_{\mathscr{P}}(-1)^{|\mathscr{P}|-1}(|\mathscr{P}|-1)!\prod_{P\in\mathscr{P}}\langle\sigma_P\rangle,
\end{equation}
where the sum is over all partitions $\mathscr{P}$ of $\{1,\dots,k\}$, $|\mathscr{P}|$ is the number of blocks in $\mathscr{P}$ (so $\mathscr{P}=\{P_1,\dots,P_{|\mathscr{P}|}\}$), and
\begin{equation}
	\langle \sigma_P\rangle:=\left\langle \prod_{j\in P}\sigma_j\right\rangle.
\end{equation}

In this paper, we are interested in spin variables $\sigma_j$ of a ferromagnetic Ising model (with only pair interactions). More precisely, let $G=(V,E)$ be a finite graph with $V$ the set of vertices and $E$ the set of edges. The Ising model on $G$ with \textit{ferromagnetic pair interactions} (or \textit{couplings}) $\bf{J}:=(J_{uv})_{uv\in E}$ where $J_{uv}\in[0,\infty)$ for each $uv\in E$ is defined by the probability measure $\mathbb{P}_G$ on $\{-1,+1\}^V$ such that
\begin{equation}\label{eq:Isingdef}
	\mathbb{P}_G(\sigma)=\frac{\exp\left[\sum_{uv\in E}J_{uv}\sigma_u\sigma_v\right]}{Z_{G}}, \sigma\in\{-1,+1\}^V,
\end{equation}
where $Z_G$ is the partition function that makes $\mathbb{P}_G$ a probability measure. The spin-flip symmetry implies that
\begin{equation}
	u_{2k-1}(\sigma_{j_1},\dots,\sigma_{j_{2k-1}})=0, \forall k\in\mathbb{N}, \forall j_1,\dots,j_{2k-1}\in V.
\end{equation}
An important result about Ursell functions due to Shlosman \cite{Shl86} is that
\begin{equation}\label{eq:Shl}
	(-1)^{k-1}u_{2k}(\sigma_{j_1},\dots,\sigma_{j_{2k}})\geq 0, \forall k\in\mathbb{N}, \forall j_1,\dots,j_{2k}\in V.
\end{equation}
We will prove that $|u_{2k}|$ is increasing in $\bf{J}$.
\begin{theorem}\label{thm:umon}
	Let $G=(V,E)$ be a finite graph and $\bf{J}$ be ferromagnetic pair interactions on $G$. Then for any $k\in\mathbb{N}$, any $j_1,\dots,j_{2k}\in V$ ($j_l$'s are not necessarily distinct), any $u_0v_0\in E$, we have
	\begin{equation}\label{eq:umon}
		(-1)^{k-1}\frac{\partial u_{2k}(\sigma_{j_1},\dots,\sigma_{j_{2k}})}{\partial J_{u_0v_0}}\geq 0.
	\end{equation}
\end{theorem}

\begin{remark}\label{rem:gen}
	Since $u_{2k}=0$ if $\bf{J}=\bf{0}$, Theorem \ref{thm:umon} implies \eqref{eq:Shl}, which is the main result of~\cite{Shl86}.
\end{remark}
Theorem \ref{thm:umon}, like \eqref{eq:Shl}, also holds for a broader class of Ising type models with 
ferromagnetic pair interactions. Namely, the underlying single spin distribution 
can be any distribution from the Griffiths-Simon class \cite{Gri69,SG73}. In particular, this includes lattice $\phi^4$ models.

\subsection{Monotonicity of the first Lee-Yang zero}
As an application of Theorem \ref{thm:umon}, we consider the partition function of the Ising model on $G$ with ferromagnetic pair interactions $\bf{J}$ and a uniform external field $h$:
\begin{equation}
	Z_{G,\bf{J},h}=\sum_{\sigma\in\{-1+1\}^V}\exp\left[\sum_{uv\in E}J_{uv}\sigma_u\sigma_v+h\sum_{u\in V}\sigma_u\right],
\end{equation}
where we added two more subscripts to $Z$ than in the previous subsection to emphasize the dependence on $\bf{J}$ and $h$. Equivalently, we may consider the moment generating function for the total magnetization of the Ising model with no external field (see \eqref{eq:Isingdef}):
\begin{equation}
	\left\langle\exp\left[h\sum_{u\in V}\sigma_u\right]\right\rangle_{G,\bf{J}}=\frac{Z_{G,\bf{J},h}}{Z_{G}}.
\end{equation}
Obviously, for fixed $G$ and $\bf{J}$, as a function $h\in\mathbb{C}$, $Z_{G,\bf{J},h}$ and $\langle\exp[h\sum_{u\in V}\sigma_u]\rangle_{G,\bf{J}}$ have the same zeros. A celebrated result due to Lee and Yang \cite{LY52} is that all zeros of $Z_{G,\bf{J},h}$ are pure imaginary; we may assume that all the zeros are $\pm i\alpha_j$ with $j\in\mathbb{N}$ such that
\begin{equation}
	0<\alpha_1\leq\alpha_2\leq\dots.
\end{equation}
See \cite{Rue71}, \cite{SG73}, \cite{New74}, \cite{LS81}, \cite{BBCKK04,BBCK04} for generalizations of the Lee-Yang result. In the original Lee-Yang program \cite{YL52,LY52}, it was argued that for suitable couplings (e.g., $J_{uv}=\beta>\beta_c$ for each $uv\in E$ where $\beta_c$ is the critical inverse temperature), as the finite system approaches the thermodynamic limit, complex singularities of the pressure (proportional to $\ln Z_{G,\bf{J},h}$) approach the physical (i.e., real) domain.
Note that $Z_{G,\bf{J},h}/\exp[h|V|]$ is a polynomial of $z:=e^{-2h}$ with degree $|V|$; 
together with spin-flip symmetry, this implies that $i\alpha_1,\dots,i\alpha_{|V|/2}$ for 
even $|V|$ ($i\alpha_1,\dots,i\alpha_{(|V|+1)/2}$ for odd $|V|$) are the principal 
zeros (i.e., other zeros are reflections or periodic translations of the principal 
ones). According to the Lee-Yang program, it seems plausible that all principal zeros 
move towards the origin as any $J_{uv}$ increases. But this conjecture turns 
out to be false \cite{NG83} (see also Remark~2 of \cite{HJN22}). Instead, Nishimori and 
Griffiths \cite{NG83} conjectured that the modulus of the first zero of $Z_{G,\bf{J},h}$, 
i.e., $\alpha_1$, is a decreasing function of each $J_{uv}$. By using the monotonicity 
of Ursell functions, we will prove this conjecture. 

\begin{theorem}\label{thm:1LY}
	Let $G=(V,E)$ be a finite graph, $\bf{J}$ and $\tilde{\bf{J}}$ be ferromagnetic pair interactions on $G$ satisfying $0\leq J_{uv}\leq\tilde{J}_{uv}$ for each $uv\in E$. Suppose that $\{\lambda_u\}_{u\in V}$ is a collection of nonnegative real numbers. Let $\alpha_1(\bf{J})$ (resp., $\alpha_1(\tilde{\bf{J}})$) be the modulus of the first zero of $\langle\exp[h\sum_{u\in V}\lambda_u\sigma_u]\rangle_{G,\bf{J}}$ (resp.,  $\langle\exp[h\sum_{u\in V}\lambda_u\sigma_u]\rangle_{G,\tilde{\bf{J}}}$). Then we have
	\begin{equation}
		\alpha_1(\bf{J})\geq\alpha_1(\tilde{\bf{J}}).
	\end{equation}
\end{theorem}

An immediate consequence of Theorem \ref{thm:1LY} by setting $\lambda_u=1$ for each $u\in V$ is the following monotonocity of the first Lee-Yang zero.
\begin{corollary}
	Let $G=(V,E)$ be a finite graph, $\bf{J}$ and $\tilde{\bf{J}}$ be ferromagnetic pair interactions on $G$ satisfying $0\leq J_{uv}\leq\tilde{J}_{uv}$ for each $uv\in E$. Let $\alpha_1(\bf{J})$ (resp., $\alpha_1(\tilde{\bf{J}})$) be the modulus of the first zero of $Z_{G,\bf{J},h}$ (resp.,  $Z_{G,\tilde{\bf{J}},h}$). Then we have
	\begin{equation}
		\alpha_1(\bf{J})\geq\alpha_1(\tilde{\bf{J}}).
	\end{equation}
\end{corollary}

\begin{remark}
	As mentioned in \cite{BBCKK04}, in spite of numerous studies about Lee-Yang zeros 
(see, e.g., \cite{BDL05} for a review and the many references therein), it seems fair 
to say that much of the original Lee-Yang program---namely, to learn about the phase 
transitions in physical systems by studying the zeros of partition functions---has 
remained unfulfilled from a rigorous point of view. The monotonicity of the first 
Lee-Yang zero serves as a step in advancing the Lee-Yang 
program as was carried out in \cite{JN23} (see also the appendix of \cite{CJN21} for a related result); an earlier paper on this topic is \cite{Tas87}. It may be worth mentioning that for some lattice models,  Lee-Yang zeros approaching the real axis do not prevent the free energy from being analytic -- see \cite{Shl86a}. Lee-Yang zero analysis was also used to show the presence of a Griffiths singularity in the dilute Ising model \cite{Gri69a,Sut82}. A further relevant reference on the analytic structure of the free energy is \cite{Isa84}, although the connection of the result in \cite{Isa84} to Lee-Yang zeros has not been determined.
\end{remark}

Since the proof of Theorem \ref{thm:1LY} (given Theorem \ref{thm:umon}) is fairly short, we include it here.
\begin{proof}[Proof of Theorem \ref{thm:1LY}]
	Let $X$ be the reweighted total magnetization
	\begin{equation}
		X:=\sum_{u\in V}\lambda_u\sigma_u.
	\end{equation}
A generalization of the Lee-Yang result (see, e.g., Theorem 1 of \cite{New75} or Lemma 4.2 of~\cite{NG83}) states that all zeros of $\langle\exp[h\sum_{u\in V}\lambda_u\sigma_u]\rangle_{G,\bf{J}}$ are pure imaginary. So
\begin{equation}\label{eq:lnmgf}
\ln\left\langle\exp\left[h\sum_{u\in V}\lambda_u\sigma_u\right]\right\rangle_{G,\bf{J}}
\text{ is analytic in }\{h\in\mathbb{C}:|h| <  \alpha_1 (\bf{J})\}.
\end{equation}
The \textit{cumulants}, $u_k(X)$, of $X$ are defined by
\begin{equation}\label{eq:cumdef}
	u_k(X):=\left.\frac{d^k \ln\left\langle\exp\left[h\sum_{u\in V}\lambda_u\sigma_u\right]\right\rangle_{G,\bf{J}} }{d h^k}\right|_{h=0};
\end{equation}
or equivalently, for small $|h|$,
\begin{equation}
	\ln\left\langle\exp\left[h\sum_{u\in V}\lambda_u\sigma_u\right]\right\rangle_{G,\bf{J}} =\sum_{k=0}^{\infty}\frac{u_k(X)}{k!} h^k.
\end{equation}
Note that by spin-flip symmetry, $u_k(X)=0$ if $k$ is odd. By \eqref{eq:lnmgf}, the last power 
series has radius of convergence $\alpha_1(\bf{J})$, and thus
\begin{equation}\label{eq:radc}
	\frac{1}{\alpha_1(\bf{J})}=\limsup_{k\rightarrow\infty}\left|\frac{u_k(X)}{k!} \right|^{1/k}.
\end{equation}
Comparing \eqref{eq:cumdef} and \eqref{eq:ursdef1}, we have
\begin{equation}\label{eq:cumurs}
	u_k(X)=\sum_{j_1,\dots,j_k\in V}\lambda_{j_1}\dots\lambda_{j_k}u_k(\sigma_{j_1},\dots,\sigma_{j_k}).
\end{equation}
Theorem \ref{thm:1LY} follows from \eqref{eq:radc}, \eqref{eq:cumurs} and Theorem \ref{thm:umon}.
\end{proof}

In Section \ref{sec:rcr}, we give a brief introduction to the random current 
representation of the Ising model; combining this with the switching lemma,  we can 
write $\partial u_{2k}/\partial J_{uv}$ in terms of one outer sum over a current and 
one inner sum over all partitions of $\{j_1,\dots,j_{2k}\}$. In Section \ref{sec:com}, 
we employ the combinatorial method introduced in \cite{Shl86} to write the  
inner sum in terms of a sum over partitions of a graph. Then we prove the latter 
sum has the desired sign by induction on the number of edges in the graph. The 
mathematical induction eventually boils the question down to special graphs where 
one vertex has degree at most~$2$ and all other vertices have degree at most $1$. 
A substantial portion of Section~\ref{sec:com} is devoted to the proof that the sum 
over partitions of each such special graph has the correct sign. In 
Section~\ref{sec:proofmain}, we use the the combinatorial result derived in 
Section~\ref{sec:com} to prove Theorem~\ref{thm:umon}, the main result of the paper.

\section{The random current representation}\label{sec:rcr}
In this section, we briefly introduce the random current representation for the Ising model and then write $\partial u_{2k}/\partial J_{uv}$ in terms of this representation. We refer to \cite{Aiz82,AF86,ADCS15} for more information about the random current representation.

Let $G=(V,E)$ be a finite graph. A  \textit{current} $\bf{n}$ on $G$ is a mapping from $E$ to $\mathbb{N}_0:=\mathbb{N}\cup\{0\}$. A \textit{source} of $\bf{n}=(n_{uv})_{uv\in E}$ is a vertex $u\in V$ at which $\sum_{v:uv\in E}\bf{n}_{uv}$ is odd. The set of all sources of $\bf{n}$ is denoted by $\partial\bf{n}$. For any $u,v\in V$, we write $u\overset{\bf{n}}{\longleftrightarrow} v$ for the event that there is a path $u=x_0,x_1,\dots,x_l=v$ such that $x_jx_{j+1}\in E$ and $\bf{n}_{x_jx_{j+1}}>0$ for each $0\leq j<l$. For each fixed $\bf{n}\in\mathbb{N}_0^E$, we define its weight by
\begin{equation}
	w(\bf{n}):=\prod_{uv\in E}\frac{J_{uv}^{\bf{n}_{uv}}}{\bf{n}_{uv}!}.
\end{equation}

The Ising partition function $Z_G$ can be written in terms of $w(\bf{n})$ as follows: 
\begin{align}
	Z_{G}&=\sum_{\sigma\in\{-1,1\}^V}\prod_{uv\in E}\exp[J_{uv}\sigma_u \sigma_v]=\sum_{\sigma\in\{-1,1\}^V}\prod_{uv\in E}\sum_{\bf{n}_{uv}=0}^{\infty}\frac{J_{uv}^{\bf{n}_{uv}}(\sigma_u \sigma_v)^{\bf{n}_{uv}}}{\bf{n}_{uv}!}\\
	&=\sum_{\bf{n}\in \mathbb{N}_0^E}w(\bf{n})\sum_{\sigma\in\{-1,1\}^V}\prod_{u\in V}\sigma_u^{\sum_{v: uv\in E}\mathbf{n}_{uv}}.
\end{align}
A direct computation gives
\begin{equation}
	\sum_{\sigma\in\{-1,1\}^V}\prod_{u\in V}\sigma_u^{\sum_{v: uv\in E}\bf{n}_{uv}}=
	\begin{cases}
		0,&\sum_{v: uv\in E}\bf{n}_{uv} \text{ is odd for some } u\in V\\
		2^{|V|},&\text{otherwise.}
	\end{cases}
\end{equation}
Therefore, we have that
\begin{equation}\label{eq:Zrcp}
Z_{G}=2^{|V|}\sum_{\bf{n}\in\mathbb{N}_0^E: \partial\bf{n}=\emptyset}w(\bf{n}).
\end{equation}
A similar expansion for correlations gives
\begin{equation}\label{eq:rcp}
	\langle \sigma_A\rangle_G=\frac{\sum_{\bf{n}\in \mathbb{N}_0^E:\partial\bf{n}=A}w(\bf{n})}{\sum_{\bf{n}\in \mathbb{N}_0^E:\partial\bf{n}=\emptyset}w(\bf{n})},\forall A\subseteq V.
\end{equation}

The following switching lemma is a key ingredient in applications of the random current representation.
\begin{lemma}[Switching Lemma]\label{lem:swi}
	Let $G=(V,E)$ be a finite graph. For any $u,v\in V$ and $A\subseteq V$, and any function $F:\mathbb{N}_0^E\rightarrow\mathbb{R}$, we have
	\begin{align}
	\sum_{\substack{\mathbf{n}\in \mathbb{N}_0^E: \partial\mathbf{n}=A\\\mathbf{m}\in \mathbb{N}_0^{E}: \partial\mathbf{m}=\{u,v\}}}w(\mathbf{n})w(\mathbf{m})F(\mathbf{n}+\mathbf{m})=\sum_{\substack{\mathbf{n}\in \mathbb{N}_0^E: \partial\mathbf{n}=A\Delta\{u,v\}\\\mathbf{m}\in \mathbb{N}_0^{E}: \partial\mathbf{m}=\emptyset}}w(\mathbf{n})w(\mathbf{m})F(\mathbf{n}+\mathbf{m})1\left[u\overset{\mathbf{n}+\mathbf{m}}{\longleftrightarrow}v \right],
\end{align}
where $A\Delta B:=(A\setminus B)\cup (B\setminus A)$ is the symmetric difference and $1[\cdot]$ is the indicator function.
\end{lemma}
\begin{proof}
	See, e.g., Lemma 2.2 of \cite{ADCS15}.
\end{proof}

We next use the random current representation to rewrite $\partial u_{2k}/\partial J_{uv}$.

\begin{lemma}\label{lem:u_2krcr}
	Let $G=(V,E)$ be a finite graph with ferromagnetic pair interactions 
$\bf{J}$ on $G$. Let $k\in\mathbb{N}$ and $j_1,\dots,j_{2k}$ be $2k$ distinct vertices in $V$. Then for any $uv\in E$,
	\begin{equation}\label{eq:udiff}
		\left(\frac{Z_G}{2^{|V|}}\right)^{k+1}\frac{\partial u_{2k}(\sigma_{j_1},\dots,\sigma_{j_{2k}})}{\partial J_{uv}}=\sum_{\substack{\bf{m}\in\mathbb{N}_0^E:\\\partial\bf{m}=\{j_1,\dots,j_{2k}\}\Delta\{u,v\}}}w(\bf{m})\sum_{\mathscr{P}}(-1)^{|\mathscr{P}|-1}(|\mathscr{P}|-1)! R(\bf{m},\mathscr{P}),
	\end{equation}
where the second sum is over all partitions 
$\mathscr{P}$ of $\{j_1,\dots,j_{2k}\}$ such that the size of each block $P_j\in \mathscr{P}$ is (strictly positive and) even, and
\begin{align}\label{eq:RmPdef}
	&R(\bf{m},\mathscr{P}):=\sum_{Q\in\mathscr{P}} R(\bf{m},\mathscr{P}, Q) \text{ with }\nonumber\\
	&R(\bf{m},\mathscr{P}, Q) :=\widetilde{\sum}\binom{\bf{m}}{\bf{n}^1,\bf{n}^2,...,\bf{n}^{k+1}}1\left[u\overset{\bf{n}^{|\mathscr{P}|}+\bf{n}^{|\mathscr{P}|+1}}{\centernot\longleftrightarrow}v\right],
\end{align}
where  the definition of $\widetilde{\sum}$ is given below, and the multinomial coefficients may be defined by
\begin{gather}
	\binom{\bf{m}}{\bf{n}^1,\bf{n}^2,...,\bf{n}^{k+1}}:=\binom{\bf{m}}{\bf{n}^{k+1}}\binom{\bf{m}-\bf{n}^{k+1}}{\bf{n}^k}\dots\binom{\bf{m}-\bf{n}^{k+1}-\dots-\bf{n}^3}{\bf{n}^2},\nonumber\\
	 \binom{\bf{m}}{\bf{n}}:=\prod_{e\in E}\binom{\bf{m}_e}{\bf{n}_e}.
\end{gather}
The sum $\widetilde{\sum}$ in \eqref{eq:RmPdef} is over $\bf{n}^1\in\mathbb{N}_0^E,\dots,\bf{n}^{k+1}\in\mathbb{N}_0^E$ such that $\{\partial\bf{n}^1,\dots,\partial\bf{n}^{|\mathscr{P}|-1}\}=\mathscr{P}\setminus\{Q\}$, $\partial \bf{n}^{|\mathscr{P}|}=Q\Delta\{u,v\}$, $\partial\bf{n}^{|\mathscr{P}|+1}=\dots=\partial\bf{n}^{k+1}=\emptyset$, and $\bf{n}^1+\dots+\bf{n}^{k+1}=\bf{m}$. E.g., if we assume $\mathscr{P}=\{P_1,\dots,P_{|\mathscr{P}|}\}$ and $Q=P_1$, then one may choose $\partial\bf{n}^1=P_2,\dots, \partial\bf{n}^{|\mathscr{P}|-1}=P_{|\mathscr{P}|}$, but the assignment for $\partial\bf{n}^1,\dots,\partial\bf{n}^{|\mathscr{P}|-1}$ is immaterial as long as it is a bijection between $\{\partial\bf{n}^1,\dots,\partial\bf{n}^{|\mathscr{P}|-1}\}$ and $\mathscr{P}\setminus\{Q\}$.
\end{lemma}
\begin{proof}
	From \eqref{eq:ursdef2}, we have
	\begin{equation}\label{eq:u_2kdiff}
		\frac{\partial u_{2k}(\sigma_{j_1},\dots,\sigma_{j_{2k}})}{\partial J_{uv}}=\sum_{\mathscr{P}}(-1)^{|\mathscr{P}|-1}(|\mathscr{P}|-1)!\sum_{Q\in\mathscr{P}}\left[\langle \sigma_Q\sigma_u\sigma_v\rangle_G-\langle\sigma_Q\rangle_G\langle\sigma_u\sigma_v\rangle_G\right]\prod_{\substack{P\in\mathscr{P}:\\P\neq Q}}\langle\sigma_P\rangle_G.
	\end{equation}
Note that by symmetry, $\langle\sigma_P\rangle_G=0$ if $|P|$ is odd. So the first sum in \eqref{eq:u_2kdiff} is really over all partitions of $\{j_1,\dots,j_{2k}\}$ such that each block $P\in\mathscr{P}$ satisfies $|P|\in 2\mathbb{N}$. The random current representation \eqref{eq:Zrcp}, \eqref{eq:rcp} and Lemma \ref{lem:swi} imply that
\begin{align}
	\left(\frac{Z_G}{2^{|V|}}\right)^2\left[\langle \sigma_Q\sigma_u\sigma_v\rangle_G-\langle\sigma_Q\rangle_G\langle\sigma_u\sigma_v\rangle_G\right]&=\sum_{\substack{\partial\bf{n}^1=Q\Delta\{u,v\}\\ \partial \bf{n}^2=\emptyset}}w(\bf{n}^1)w(\bf{n}^2)-\sum_{\substack{\partial\bf{n}^1=Q\\ \partial\bf{n}^2=\{u,v\}}}w(\bf{n}^1)w(\bf{n}^2)\nonumber\\
	&=\sum_{\substack{\partial\bf{n}^1=Q\Delta\{u,v\}\\ \partial \bf{n}^2=\emptyset}}w(\bf{n}^1)w(\bf{n}^2)\left[1-1\left[u\overset{\bf{n}^1+\bf{n}^2}{\longleftrightarrow}v\right]\right]\nonumber\\
	&=\sum_{\substack{\partial\bf{n}^1=Q\Delta\{u,v\}\\ \partial \bf{n}^2=\emptyset}}w(\bf{n}^1)w(\bf{n}^2)1\left[u\overset{\bf{n}^1+\bf{n}^2}{\centernot\longleftrightarrow}v\right].
\end{align}
A further application of \eqref{eq:Zrcp} and \eqref{eq:rcp} gives
\begin{align}\label{eq:diffu_2k}
	&\left(\frac{Z_G}{2^{|V|}}\right)^{k+1}\frac{\partial u_{2k}(\sigma_{j_1},\dots,\sigma_{j_{2k}})}{\partial J_{uv}}=\sum_{\mathscr{P}}(-1)^{|\mathscr{P}|-1}(|\mathscr{P}|-1)!\sum_{Q\in\mathscr{P}}\Big\{\nonumber\\
	&\qquad\sum_{Q\Delta\{u,v\},\emptyset,\mathscr{P}\setminus Q}w(\bf{n}^1)w(\bf{n}^2)1\left[u\overset{\bf{n}^1+\bf{n}^2}{\centernot\longleftrightarrow}v\right]w(\bf{n}^3)\dots w(\bf{n}^{k+1})\Big\},
\end{align}
where the  sum in the curly brackets is over currents $\bf{n}^1\in\mathbb{N}_0^E,\dots,\bf{n}^{|\mathscr{P}|+1}\in\mathbb{N}_0^E$ such that $\partial\bf{n}^1=Q\Delta\{u,v\},\partial\bf{n}^2=\emptyset,\{\partial \bf{n}^3,\dots,\partial\bf{n}^{|\mathscr{P}|+1}\}=\mathscr{P}\setminus Q$, $\partial\bf{n}^{|\mathscr{P}|+2}=\dots=\partial\bf{n}^{k+1}=\emptyset$.

We perform the following changes of variables in \eqref{eq:diffu_2k},
\begin{align}
	&\bf{m}=\bf{n}^1+\dots+\bf{n}^{k+1},\tilde{\bf{n}}^2=\bf{n}^4,\tilde{\bf{n}}^3=\bf{n}^5,\dots,\tilde{\bf{n}}^{|\mathscr{P}|-1}=\bf{n}^{|\mathscr{P}|+1},\tilde{\bf{n}}^{|\mathscr{P}|}=\bf{n}^{1},\nonumber\\
	&\tilde{\bf{n}}^{|\mathscr{P}|+1}=\bf{n}^{2}, \tilde{\bf{n}}^{|\mathscr{P}|+2}=\bf{n}^{|\mathscr{P}|+2},\dots,\tilde{\bf{n}}^{k+1}=\bf{n}^{k+1}.
\end{align}
So in the curly brackets of \eqref{eq:diffu_2k}, we skip $\bf{n}^3$, and then reorder $\bf{n}^1,\bf{n}^2,\bf{n}^4,\dots,\bf{n}^{k+1}$.
These changes of variables complete the proof of the lemma.
\end{proof}

For each $\bf{m}\in \mathbb{N}_0^E$, we define
\begin{equation}\label{eq:Rmdef}
	R(\bf{m}):=\sum_{\mathscr{P}}(-1)^{|\mathscr{P}|-1}(|\mathscr{P}|-1)! R(\bf{m},\mathscr{P}),
\end{equation}
which is the second sum in \eqref{eq:udiff}.
Then a consequence of Lemma \ref{lem:u_2krcr} is 
\begin{lemma}\label{lem:Rmtod}
	For each $\bf{m}\in \mathbb{N}_0^E$, we have
	\begin{equation}
		R(\bf{m})=\left.\frac{\partial^{\bf{m}_{e_1}+\dots+\bf{m}_{e_{|E|}}}}{\partial J^{\bf{m}_{e_1}}_{e_1}\dots\partial J^{\bf{m}_{e_{|E|}}}_{e_{|E|}}}\left[\left(\frac{Z_G}{2^{|V|}}\right)^{k+1}\frac{\partial u_{2k}(\sigma_{j_1},\dots,\sigma_{j_{2k}})}{\partial J_{uv}}\right]\right|_{\bf{J}=\bf{0}}.
	\end{equation}
\end{lemma}
\begin{proof}
	Lemma \ref{lem:u_2krcr} gives that
	\begin{align}
		&\left(\frac{Z_G}{2^{|V|}}\right)^{k+1}\frac{\partial u_{2k}(\sigma_{j_1},\dots,\sigma_{j_{2k}})}{\partial J_{uv}}\nonumber\\
		&\quad= \sum_{\substack{\tilde{\bf{m}}\in\mathbb{N}_0^E:\\\partial\tilde{\bf{m}}=\{j_1,\dots,j_{2k}\}\Delta\{u,v\}}}\prod_{e\in E}\frac{J_e^{\tilde{\bf{m}}_e}}{\tilde{\bf{m}}_e!}\sum_{\mathscr{P}}(-1)^{|\mathscr{P}|-1}(|\mathscr{P}|-1)!R(\tilde{\bf{m}},\mathscr{P}).
	\end{align}
Note that
\begin{equation}
	\left.\frac{\partial^{\bf{m}_{e_1}+\dots+\bf{m}_{e_{|E|}}}}{\partial J^{\bf{m}_{e_1}}_{e_1}\dots\partial J^{\bf{m}_{e_{|E|}}}_{e_{|E|}}}\prod_{e\in E}\frac{J_e^{\tilde{\bf{m}}_e}}{\tilde{\bf{m}}_e!}\right|_{\bf{J}=\bf{0}}=\begin{cases}
		1, &\bf{m}_{e}=\tilde{\bf{m}}_e, \forall e\in E\\
		0, &\text{otherwise.}
	\end{cases}
\end{equation}
The lemma follows from \eqref{eq:Rmdef} and the last two displayed equations.
\end{proof}

\section{A combinatorial result}\label{sec:com}
In this section, we use the combinatorial method introduced by Shlosman in \cite{Shl86} to derive the sign of $R(\bf{m})$ defined in \eqref{eq:Rmdef}. Our main result is the following proposition.
\begin{proposition}\label{prop:Rmsgn}
	Let $G=(V,E)$ be a finite graph. Let $k\in\mathbb{N}$ and $j_1,\dots,j_{2k}$ be $2k$ distinct vertices in $V$. Then for any $u_0v_0\in E$ satisfying $v_0\notin\{j_1,\dots,j_{2k}\}$ and any $\bf{m}\in\mathbb{N}_0^{E}$ with $\partial\bf{m}=\{j_1,\dots,j_{2k}\}\Delta\{u_0,v_0\}$, we have
	\begin{equation}
		(-1)^{k-1} R(\bf{m})\geq 0,
	\end{equation}
where $R(\bf{m})$ is defined in \eqref{eq:Rmdef}.
\end{proposition}

The structure of our proof of Proposition~\ref{prop:Rmsgn}, given in this section,
is as follows. In the first subsection (see \eqref{eq:RGdef} - \eqref{eq:RG=Rm}), we show that 
$R(\bf{m})$ equals a quantity $R(\mathcal{G})$ with $\mathcal{G}$ a graph associated to $\bf{m}$.
Then in the second subsection, the desired inequality for $R(\mathcal{G})$ is reduced to the cases of
three special types of graphs $\mathcal{G}$ which are treated in the third subsection.

\subsection{A representation in terms of partitions of a graph}
As in \cite{Shl86}, we first write $R(\bf{m})$ in terms of a sum over partitions of a graph. To each $\bf{m}\in\mathbb{N}_0^E$, we associate a graph $\mathcal{G}=(V,\mathcal{E})$ with $V$ the same as the set of vertices in $G$, and $\bf{m}_e$ (undirected) \textit{labelled} edges between $u$ and $v$ for each $e=uv\in E$ . We note that, if $\bf{m}_e>1$ for some $e\in E$, then $\mathcal{G}$ is a multigraph. We call a graph $\Gamma_1=(V,\mathcal{E}_1)$ a \textit{subgraph} of $\mathcal{G}$ if it has the same set of vertices $V$ and $\mathcal{E}_1\subseteq\mathcal{E}$. For $v\in V$, let $\deg_{\Gamma_1}(v)$ be the number of edges in $\mathcal{E}_1$ incident on $v$. We define
\begin{equation}
	\partial \Gamma_1:=\{v\in V: \deg_{\Gamma_1}(v) \text{ is odd}\}.
\end{equation}
In particular, we have $\partial \mathcal{G}=\partial \bf{m}=\{j_1,\dots,j_{2k}\}\Delta\{u_0,v_0\}$.

A \textit{partition $\mathscr{T}$ of the graph $\mathcal{G}$} is a family of subgraphs of $\mathcal{G}$, $\Gamma_1=(V,\mathcal{E}_1), \dots, \Gamma_{k+1}=(V,\mathcal{E}_{k+1})$, such that
\begin{enumerate}
	\item there exist $n(\mathscr{T})\in \mathbb{N}$ and a partition $\{P_1,\dots,P_{n(\mathscr{T})}\}$ of $\{j_1,\dots,j_{2k}\}$ satisfying: $\cup_{j=1}^{n(\mathscr{T})}P_j=\{j_1,\dots,j_{2k}\}$, $|P_j|\in 2\mathbb{N}$ for all $j$, $P_j\cap P_l=\emptyset$ for all $j\neq l$,
	\item $\mathcal{E}_1,\dots,\mathcal{E}_{k+1}$ satisfy: $\cup_{j=1}^{k+1}\mathcal{E}_j=\mathcal{E}$, $\mathcal{E}_j\cap\mathcal{E}_l=\emptyset$ for all $j\neq l$,
	\item $\partial \Gamma_j=P_j$ for $1\leq j\leq n(\mathscr{T})-1$, $\partial \Gamma_{n(\mathscr{T})}=P_{n(\mathscr{T})}\Delta\{u_0,v_0\}$, $\partial\Gamma_j=\emptyset$ for $n(\mathscr{T})+1\leq j\leq k+1$,
	\item $u_0\overset{\mathcal{E}_{n(\mathscr{T})}+\mathcal{E}_{n(\mathscr{T})+1}}{\centernot\longleftrightarrow}v_0$,
\end{enumerate}
where the last item means that $u_0$ and $v_0$ are not connected in the graph $(V,\mathcal{E}_{n(\mathscr{T})}\cup\mathcal{E}_{n(\mathscr{T})+1})$.

Two partitions $\mathscr{T}=(\Gamma_1,\dots,\Gamma_{k+1})$ and $\tilde{\mathscr{T}}=(\tilde{\Gamma}_1,\dots,\tilde{\Gamma}_{k+1})$ of $\mathcal{G}$ are \textit{distinct} if
\begin{enumerate}
	\item as a set, $\{\partial\Gamma_1,\dots,\partial\Gamma_{n(\mathscr{T})}\}\neq\{\partial\tilde{\Gamma}_1,\dots,\partial\tilde{\Gamma}_{n(\tilde{\mathscr{T}})}\}$, i.e., for some $1\leq j\leq n(\mathscr{T})$, $\partial\Gamma_j\neq\partial\tilde{\Gamma}_l$ for all $l$; otherwise if
	\item for some $j, l$ with $1\leq j, l\leq n(\mathscr{T})=n(\tilde{\mathscr{T}})$,  $\partial\Gamma_j=\partial\tilde{\Gamma}_l$ but $\Gamma_j\neq\tilde{\Gamma}_l$; otherwise if
	\item as a tuple (i.e., an ordered list), $(\Gamma_{n(\mathscr{T})+1},\dots,\Gamma_{k+1})\neq (\tilde{\Gamma}_{n(\tilde{\mathscr{T}})+1},\dots,\tilde{\Gamma}_{k+1})$, i.e., for some $j>n(\mathscr{T})=n(\tilde{\mathscr{T}})$, $\Gamma_j\neq\tilde{\Gamma}_j$.
\end{enumerate}

We define
\begin{equation}\label{eq:RGdef}
	R(\mathcal{G}):=\sum_{\mathscr{T}}(-1)^{n(\mathscr{T})-1}(n(\mathscr{T})-1)!,
\end{equation}
where the sum is over all distinct partitions $\mathscr{T}$ of $\mathcal{G}$. It is clear from these definitions and Lemma \ref{lem:u_2krcr} (by viewing each $\Gamma_j$ as the graph associated with $\bf{n}^j$) that 
\begin{equation}\label{eq:RG=Rm}
	R(\mathcal{G})=R(\bf{m}) \text{ if } \mathcal{G} \text{ is the graph associated with } \bf{m}. 
\end{equation}

\subsection{Reduction to special graphs}
\subsubsection{An initial reduction}\label{subsubsec:initial}
Equation \eqref{eq:RG=Rm} implies that Proposition \ref{prop:Rmsgn} follows directly from the following proposition.
\begin{proposition}\label{prop:stronger}
	Let $G=(V,E)$ be a finite graph. Let $k\in\mathbb{N}$ and $j_1,\dots,j_{2k}$ be $2k$ distinct vertices in $V$. Suppose $u_0v_0\in E$ satisfies $v_0\notin\{j_1,\dots,j_{2k}\}$. For any $\bf{m}\in\mathbb{N}_0^E$ with $\partial\bf{m}=\{j_1,\dots,j_{2k}\}\Delta\{u_0,v_0\}$, let $\mathcal{G}=(V,\mathcal{E})$ be the graph associated with $\bf{m}$. For any sequence of pairs of edges in $\mathcal{E}$,
	\begin{equation}
		(e_{k_1},e_{l_1}),\dots,(e_{k_n},e_{l_n}), n\in\mathbb{N},
	\end{equation}
let $R(\mathcal{G};\{e_{k_1},e_{l_1}\},\dots,\{e_{k_n},e_{l_n}\})$ be defined as \eqref{eq:RGdef} with the sum restricted to partitions $\mathscr{T}$ satisfying that $e_{k_p}$ and $e_{l_p}$ belong to different subgraphs for each $p=1,\dots,n$. Then we have
\begin{align}
	&(-1)^{k-1}R(\mathcal{G})\geq0,\label{eq:RGsgn}\\
	&(-1)^{k-1}R(\mathcal{G};\{e_{k_1},e_{l_1}\},\dots,\{e_{k_n},e_{l_n}\})\geq0 \text{ provided no edge between}\nonumber\\
	& u_0 \text{ and }  v_0 \text{ is in } \bigcup_{p=1}^n\{e_{k_p}, e_{l_p}\} \text{ and }\bigcup_{p=1}^n\{e_{k_p}, e_{l_p}\} \text{ contains no self-loop},\label{eq:RGmsgn}
\end{align}
where a self-loop means an edge that connects a vertex to itself.
\end{proposition}

\begin{remark}
	Here we allow self-loops in $\mathcal{G}$ since a self-loop may show up when we reduce $\mathcal{G}$ even if $\mathcal{G}$ has no self-loop; e.g., $e_{12}$ defined in \eqref{eq:tildeG} below could be a self-loop.
	One may wonder if the conditions in \eqref{eq:RGmsgn} are redundant, the following examples show that they are not. Let $\mathcal{H}_{2,1}$ be the graph defined in Figure \ref{fig:H_kL} below with $e_1$ and $e_2$ being the top and bottom edge between $u_0$ and $v_0$ respectively. Then the only partition of $\mathcal{H}_{2,1}$ satisfying the restrictions $\{e_1,j_1j_2\}$ and $\{e_2,j_1j_3\}$ is 
	\begin{equation}
		\mathcal{E}_1=\{j_1j_2,j_1j_3,j_4v_0\}, \mathcal{E}_2=\emptyset, \mathcal{E}_3=\{e_1,e_2\} \text{ with }n(\mathscr{T})=1.
	\end{equation}
To see that this is the only partition, one can simply list all partitions of $\mathcal{H}_{2,1}$.
Therefore,
\begin{equation}
	(-1)^{2-1}R(\mathcal{H}_{2,1};\{e_1,j_1j_2\},\{e_2,j_1j_3\})=-1.
\end{equation}
Let $\mathcal{K}_2^I$ be the graph defined in Figure \ref{fig:K_pI} below. We add a self-loop $j_1j_1$ to $\mathcal{K}_2^I$, and call the new graph $\hat{\mathcal{K}}_2^I$. Then there are exactly three partitions of $\hat{\mathcal{K}}_2^I$ satisfying the restrictions $\{j_1j_1,j_2v_0\}$ and $\{j_1j_1,j_3j_4\}$:
\begin{enumerate}
	\item
	$\mathcal{E}_1=\{j_2v_0, j_3j_4\}, \mathcal{E}_2=\{j_1j_1\},\mathcal{E}_3=\emptyset$ with $n(\mathscr{T})=1$,
	\item 
	$\mathcal{E}_1=\{j_2v_0, j_3j_4\} , \mathcal{E}_2=\emptyset, \mathcal{E}_3=\{j_1j_1\}$ with $n(\mathscr{T})=1$,
	\item
	$\mathcal{E}_1=\{ j_3j_4\}, \mathcal{E}_2=\{j_2v_0\},\mathcal{E}_3=\{j_1j_1\}$ with $n(\mathscr{T})=2$.
\end{enumerate}
Therefore,
\begin{equation}
	(-1)^{2-1}R(\hat{\mathcal{K}}_2^I;\{j_1j_1,j_2v_0\},\{j_1j_1,j_3j_4\})=-1.
\end{equation}
\end{remark}

\begin{proof}[Proof of Proposition \ref{prop:stronger}]
	Note that $G=(V,E)$, $k\in\mathbb{N}$, $\{j_1,\dots,j_{2k}\}\subset V$ and $u_0v_0\in E$ are all fixed. We first prove \eqref{eq:RGsgn} by induction on the number of edges in $\mathcal{G}$, i.e., on $|\mathcal{E}|$. Equivalently, we perform induction on $|\bf{m}|:=\sum_{e\in E}\bf{m}_e$.
	
	Our assumption $v_0\notin\{j_1,\dots,j_{2k}\}$ implies
	\begin{equation}
		\left|\{j_1,\dots,j_{2k}\}\Delta\{u_0,v_0\}\right|\geq\left|\{v_0\}\right|=1.
	\end{equation}
So we must have $|\bf{m}|=|\mathcal{E}|\geq 1$.

The base case of the induction is the set of graphs defined in \eqref{eq:Gclass1}, and will be proved by the combination of Lemma \ref{lem:red} and Proposition \ref{prop:RGsgn}. Some readers may prefer to read the arguments for the base case first.

Suppose \eqref{eq:RGsgn} holds for any $\mathcal{G}=(V,\mathcal{E})$ with $\partial\mathcal{G}=\{j_1,\dots,j_{2k}\}\Delta\{u_0,v_0\}$, $|\mathcal{E}|\leq N$ where $N\in\mathbb{N}$.

We now consider a new $\mathcal{G}=(V,\mathcal{E})$ with $\partial\mathcal{G}=\{j_1,\dots,j_{2k}\}\Delta\{u_0,v_0\}$ and $|\mathcal{E}|= N+1$. Suppose that there is a self-loop in $\mathcal{G}$, i.e., an edge $e=vv$ with $v\in V$. We may define a new graph $\hat{\mathcal{G}}=(\hat{V},\hat{\mathcal{E}})$ by
\begin{equation}\label{eq:hatG}
	\hat{V}:=V, \hat{\mathcal{E}}:=\mathcal{E}\setminus\{vv\}.
\end{equation}
For each partition of $\hat{\mathcal{G}}$, there are exactly $k+1$ corresponding partitions of $\mathcal{G}$ since the self-loop $vv$ can be added to any one of the $k+1$ subgraphs $\hat{\mathcal{G}}$; each partition of $\mathcal{G}$ can be obtained in this way. This implies that
\begin{equation}
	R(\mathcal{G})=(k+1)R(\hat{\mathcal{G}}).
\end{equation}
So whenever there is a self-loop, we may just delete the self-loop and apply the induction hypothesis. Hence, we may assume that $\mathcal{G}$ contains no self-loop.

If there is some $v\in V\setminus\{u_0,v_0\}$ with $\deg_{\mathcal{G}}(v)\geq 2$, i.e., there are $e_1=vv_1$ and $e_2=vv_2$ in $\mathcal{E}$ such that $e_1\neq e_2$ (note that it is possible that $v_1=v_2$), then we may write
\begin{equation}\label{eq:RGdec}
	R(\mathcal{G})=R(\mathcal{G};[e_1,e_2])+R(\mathcal{G};\{e_1,e_2\}),
\end{equation}
where the first term on the RHS is defined as \eqref{eq:RGdef} with the sum restricted to all partitions with both $e_1$ and $e_2$ contained in the same subgraph; the second term has already been defined above \eqref{eq:RGsgn}. We define a new graph $\tilde{\mathcal{G}}=(\tilde{V},\tilde{\mathcal{E}})$ by
\begin{equation}\label{eq:tildeG}
	\tilde{V}:=V, \tilde{\mathcal{E}}:=\left(\mathcal{E}\setminus\{e_1,e_2\}\right)\cup \{e_{12}\},
\end{equation}
where $e_{12}=v_1v_2$. Then it is obvious that
\begin{equation}
	R(\mathcal{G};[e_1,e_2])=R(\tilde{\mathcal{G}}), |\tilde{\mathcal{E}}|=|\mathcal{E}|-1=N.
\end{equation}
So the induction hypothesis implies that
\begin{equation}
	(-1)^{k-1}	R(\mathcal{G};[e_1,e_2])=(-1)^{k-1}R(\tilde{\mathcal{G}})\geq 0.
\end{equation}
The second term on the RHS of \eqref{eq:RGdec} will also have the desired sign if we can prove \eqref{eq:RGmsgn}. Before we embark on the proof of \eqref{eq:RGmsgn}, we note that we have not finished the proof of \eqref{eq:RGsgn} even if we are given \eqref{eq:RGmsgn}. What is missing is that we still need to check the special case: \eqref{eq:RGsgn} holds for each $\mathcal{G}$ with no self-loop, $\partial\mathcal{G}=\{j_1,\dots,j_{2k}\}\Delta\{u_0,v_0\}$, and $\deg_{\mathcal{G}}(v)\leq 1$ for every $v\in V\setminus\{u_0,v_0\}$.  We will deal with this special case in \eqref{eq:Gclass}, and then reduce it further to the base case \eqref{eq:Gclass1}.

We proceed to prove \eqref{eq:RGmsgn} by induction on $|\mathcal{E}|$.

Once again, the base case of the induction is the set of graphs defined in \eqref{eq:Gclass1}, and will be proved by the combination of Lemma \ref{lem:red} and Proposition \ref{prop:RGsgn}.

Suppose \eqref{eq:RGmsgn} holds for any $\mathcal{G}=(V,\mathcal{E})$ with $\partial\mathcal{G}=\{j_1,\dots,j_{2k}\}\Delta\{u_0,v_0\}$, $|\mathcal{E}|\leq N$ where $N\in\mathbb{N}$, and any pairs of edges $(e_{k_1},e_{l_1}),\dots,(e_{k_n},e_{l_n})$ from $\mathcal{E}$ with $n\in\mathbb{N}$ satisfying that none of the $2n$ edges is an edge between $u_0$ and $v_0$ and none of them is a self-loop.

We now consider a new $\mathcal{G}=(V,\mathcal{E})$ with $\partial\mathcal{G}=\{j_1,\dots,j_{2k}\}\Delta\{u_0,v_0\}$ and $|\mathcal{E}|= N+1$, and $e_{k_1},\dots, e_{k_n}$, $e_{l_1},\dots,e_{l_n}$ are given edges from $\mathcal{E}$ satisfying that none of them is an edge between $u_0$ and $v_0$ and none of them is a self-loop.  Suppose that there is a self-loop $vv$ in $\mathcal{G}$. Then we know $vv\neq e_{k_p}$, $vv\neq e_{l_p}$ for each $1\leq p\leq n$. Let $\hat{\mathcal{G}}$ be the graph defined by \eqref{eq:hatG}. For each partition of $\hat{\mathcal{G}}$ satisfying $\{e_{k_1},e_{l_1}\},\dots,\{e_{k_n},e_{l_n}\}$, we can also add $vv$ to any one of the $k+1$ subgraphs of $\hat{\mathcal{G}}$; each partition of $\mathcal{G}$ can be obtained in this way. Thus
\begin{equation}
	R(\mathcal{G};\{e_{k_1},e_{l_1}\},\dots,\{e_{k_n},e_{l_n}\})=(k+1)R(\hat{\mathcal{G}};\{e_{k_1},e_{l_1}\},\dots,\{e_{k_n},e_{l_n}\}).
\end{equation}
So whenever there is a self-loop, we may just delete the self-loop and apply the induction hypothesis. Hence, we may assume that $\mathcal{G}$ contains no self-loop.

If there is some $v\in V\setminus\{u_0,v_0\}$ with $e_1=vv_1, e_2=vv_2\in \mathcal{E}$ where $e_1\neq e_2$ such that as an unordered pair $\{e_1,e_2\}\neq \{e_{k_p},e_{l_p}\}$ for each $p=1,\dots,n$, then we may write
\begin{align}\label{eq:RGdec1}
	R(\mathcal{G};\{e_{k_1},e_{l_1}\},\dots,\{e_{k_n},e_{l_n}\})=&R(\mathcal{G};[e_1,e_2],\{e_{k_1},e_{l_1}\},\dots,\{e_{k_n},e_{l_n}\})+\nonumber\\
	&R(\mathcal{G};\{e_1,e_2\},\{e_{k_1},e_{l_1}\},\dots,\{e_{k_n},e_{l_n}\}),
\end{align}
where the first term on the RHS is defined as \eqref{eq:RGdef} with the sum restricted to all partitions satisfying: $e_1$ and $e_2$ are in the same subgraph, $e_{k_p}$ and $e_{l_p}$ are in different subgraphs for each $p=1,\dots,n$; the second term is already defined before. The first term can be dealt with the new graph $\tilde{\mathcal{G}}=(\tilde{V},\tilde{\mathcal{E}})$ defined in \eqref{eq:tildeG}, and then the induction hypothesis implies that
\begin{equation}
	(-1)^{k-1}R(\mathcal{G};[e_1,e_2],\{e_{k_1},e_{l_1}\},\dots,\{e_{k_n},e_{l_n}\})\geq0.
\end{equation}

We apply the formula \eqref{eq:RGdec1} repeatedly whenever we can add a new pair of edges with a common vertex in $V\setminus\{u_0,v_0\}$ into the list of existing pairs. The resulted first term will always have the desired sign by the induction hypothesis. The number of pairs of edges in the second term increases by one after each application of \eqref{eq:RGdec1}. The process stops when each pair of incident edges of $\mathcal{G}$ with a common vertex in $V\setminus\{u_0,v_0\}$ is in the list. We denote this final term by $R^*(\mathcal{G};\dots)$.

\begin{claim}\label{cla:red}
	If there is some $v\in V\setminus\{u_0,v_0\}$ such that $\deg_{\mathcal{G}}(v)\geq 2$, then
	\begin{equation}
		R^*(\mathcal{G};\dots)=0.
	\end{equation}
\end{claim}
\begin{proof} [Proof of the Claim]
	Suppose there is some $v\in V\setminus\{u_0,v_0\}$ such that $\deg_{\mathcal{G}}(v)\geq 2$, i.e., there are $vv_1=e_1\neq e_2=vv_2\in\mathcal{E}$. By our definition of $R^*$, $vv_1$ and $vv_2$ must be in different subgraphs (say $\Gamma_1$ and $\Gamma_2$ respectively) and each subgraph in a fixed partition contains at most one edge incident on $v$. This means that $v\in\partial \Gamma_1$ and $v\in\partial\Gamma_2$, which violates our definition of partitions. So there is no partition in this case. 
\end{proof}
 This claim implies that we only need to prove \eqref{eq:RGsgn} and \eqref{eq:RGmsgn} for the following set of graphs:
 \begin{align}\label{eq:Gclass}
 	\{\mathcal{G}: &\mathcal{G} \text{ has no self-loop}, \partial\mathcal{G}=\{j_1,\dots,j_{2k}\}\Delta\{u_0,v_0\}, v_0\notin\{j_1,\dots,j_{2k}\},\nonumber\\
 	 &\deg_{\mathcal{G}}(v)\leq 1~\forall v \in V\setminus\{u_0,v_0\}\}.
 \end{align}

\subsubsection{A further reduction}
In Subsubsection \ref{subsubsec:initial}, we managed to reduce the degree of $v$ for each $v\in V\setminus\{u_0,v_0\}$ (see \eqref{eq:Gclass}). A typical graph resulted from this initial reduction can be found in Figure \ref{fig:G_k}. 
\begin{figure}
	\begin{center}
		\includegraphics{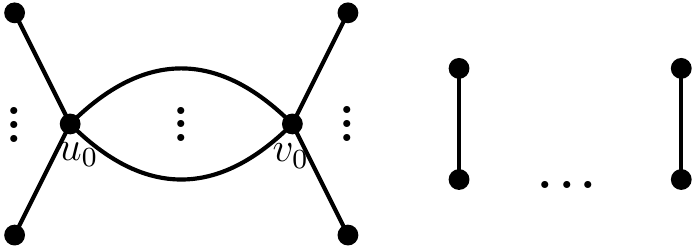}
		\caption{A typical graph after the initial reduction}\label{fig:G_k}
	\end{center}
\end{figure}
In the current subsubsection, we will reduce the degrees of $u_0$ and $v_0$. It turns out that $u_0$ and $v_0$ are more subtle because of the restriction $u_0\centernot{\longleftrightarrow}v_0$ in $(V,\mathcal{E}_{n(\mathscr{T})}\cup\mathcal{E}_{n(\mathscr{T})+1})$ in the definition of a partition of $\mathcal{G}$. We first observe that the following sets of edges
\begin{equation}\label{eq:adju0v0}
	\{e\in\mathcal{E}:e=u_0v \text{ with }v\neq v_0\}, \{e\in\mathcal{E}:e=v_0v \text{ with }v\neq u_0\}
\end{equation}
do not affect the event $u_0\longleftrightarrow v_0$ in $\mathcal{G}$ for each $\mathcal{G}$ from \eqref{eq:Gclass} (e.g., if $e=u_0v$ with $v\neq v_0$ then the only edge incident on $v$ is $u_0v$). So we may continue applying the induction method in Subsubsection \ref{subsubsec:initial} to those $\mathcal{G}$'s defined in \eqref{eq:Gclass}. More precisely, we apply \eqref{eq:RGdec} and \eqref{eq:RGdec1} to those $e_1$ and $e_2$ from \eqref{eq:adju0v0}. Then we end up with a smaller set of graphs:
\begin{align}\label{eq:Gclass1}
	\{\mathcal{G}: &\mathcal{G} \text{ has no self-loop}, \partial\mathcal{G}=\{j_1,\dots,j_{2k}\}\Delta\{u_0,v_0\}, v_0\notin\{j_1,\dots,j_{2k}\}, \nonumber\\ &\deg_{\mathcal{G}\setminus v_0}(u_0)\leq 2, \deg_{\mathcal{G}\setminus u_0}(v_0)=1, \deg_{\mathcal{G}}(v)\leq 1~\forall v\in V\setminus\{u_0,v_0\}\},
\end{align}
where $\mathcal{G}\setminus v_0$ is the graph obtained from $\mathcal{G}$ by removing $v_0$ and all edges incident on $v_0$; a similar argument as in Claim \ref{cla:red} implies that $\deg_{\mathcal{G}\setminus v_0}(u_0)\leq 1$ if $u_0\notin\{j_1,\dots,j_{2k}\}$ since $u_0$ is in $\partial \Gamma_{n(\mathscr{T})}$ only, and $\deg_{\mathcal{G}\setminus v_0}(u_0)\leq 2$ if $u_0\in\{j_1,\dots,j_{2k}\}$ since it is possible that $u_0$ is in $\partial \Gamma_{n(\mathscr{T})}$ and another $\partial \Gamma_{j}$ for some $1\leq j\leq n(\mathscr{T})-1$; $\deg_{\mathcal{G}\setminus u_0}(v_0)\leq1$ follows from a similar argument as in Claim \ref{cla:red} and we only need to consider $\deg_{\mathcal{G}\setminus u_0}(v_0)=1$ since otherwise there is no partition of $\mathcal{G}$ (note that $v_0\in\partial\Gamma_{n(\mathscr{T})}$ and $u_0\centernot{\longleftrightarrow}v_0$ in $(V,\mathcal{E}_{n(\mathscr{T})}\cup\mathcal{E}_{n(\mathscr{T})+1})$).

It is clear that the definitions of $R(\mathcal{G})$ and $R(\mathcal{G};\{e_{k_1},e_{l_1}\},\dots,\{e_{k_n},e_{l_n}\})$ do not depend on any isolated vertex in $V\setminus \{u_0\}$. So without loss of generality, we may assume that the graphs defined in \eqref{eq:Gclass1} contain no isolated vertex other than $u_0$. Then all graphs in \eqref{eq:Gclass1} can be divided into the following three families of graphs:
\begin{enumerate}
	\item $\deg_{\mathcal{G}\setminus v_0}(u_0)= 2$, $\deg_{\mathcal{G}\setminus u_0}(v_0)=1$, $\deg_{\mathcal{G}}(v)= 1$ for each $v\in V\setminus\{u_0,v_0\}$, there are exactly $2L$ edges between $u_0$ and $v_0$ with $L\in\mathbb{N}\cup\{0\}$. The number of edges between $u_0$ and $v_0$ is even because $v_0\notin\{j_1,\dots,j_{2k}\}$ and $\partial\mathcal{G}=\{j_1,\dots,j_{2k}\}\Delta\{u_0,v_0\}$ imply that $v_0\in\partial\mathcal{G}$ (and thus $\deg_{\mathcal{G}}(v_0)$ should be odd). Here we have $k\geq 2$ since $\deg_{\mathcal{G}\setminus v_0}(u_0)= 2$ and $\deg_{\mathcal{G}\setminus u_0}(v_0)=1$ imply that there are at least 3 vertices in $\{j_1,\dots,j_{2k}\}$. We call this graph $\mathcal{H}_{k,L}$ with $k\geq 2$; see Figure \ref{fig:H_kL}. Note that $u_0\in\{j_1,\dots,j_{2k}\}$ since otherwise $u_0\in\partial\mathcal{G}=\{j_1,\dots,j_{2k}\}\Delta\{u_0,v_0\}$ which contradicts the evenness of $\deg_{\mathcal{G}}(u_0)$. Without loss of generality, we may assume $u_0=j_1$.  We note that the labeling for $j_1,\dots,j_{2p}$ are immaterial.
	\begin{figure}
		\begin{center}
			\includegraphics{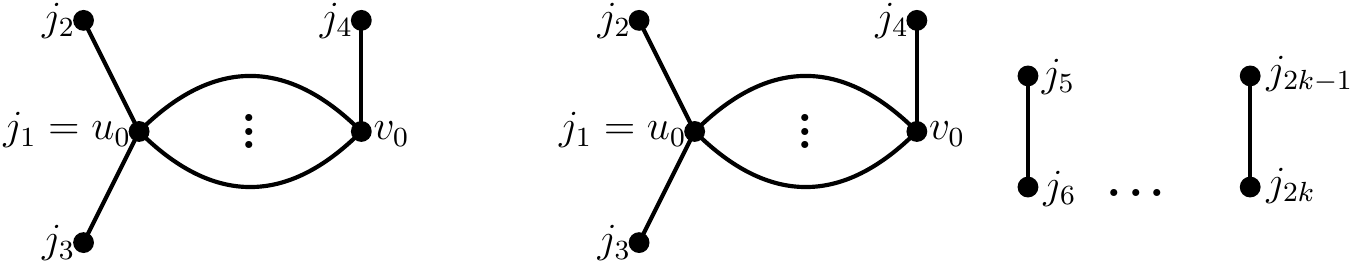}
			\caption{$\mathcal{H}_{2,L}$ left, $\mathcal{H}_{k,L}$ with $k\geq3$ right; each graph has $2L$ edges between $u_0$ and $v_0$}\label{fig:H_kL}
		\end{center}
	\end{figure}
	\item $\deg_{\mathcal{G}\setminus v_0}(u_0)= 0$, $\deg_{\mathcal{G}\setminus u_0}(v_0)=1$, $\deg_{\mathcal{G}}(v)= 1$ for each $v\in V\setminus\{u_0,v_0\}$, there are exactly $2L$ edges between $u_0$ and $v_0$ with $L\in\mathbb{N}\cup\{0\}$. We call this graph $\mathcal{K}_{k,L}^I$ with $k\geq 1$; see Figure \ref{fig:K_kLI}.
	\begin{figure}
		\begin{center}
			\includegraphics{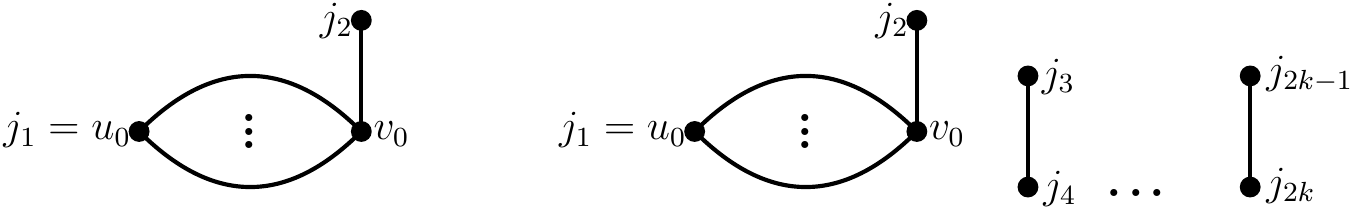}
			\caption{$\mathcal{K}_{1,L}^I$ left, $\mathcal{K}_{k,L}^I$ with $k\geq2$ right; each graph has $2L$ edges between $u_0$ and $v_0$}\label{fig:K_kLI}
		\end{center}
	\end{figure}
	\item $\deg_{\mathcal{G}\setminus v_0}(u_0)= 1$, $\deg_{\mathcal{G}\setminus u_0}(v_0)=1$, $\deg_{\mathcal{G}}(v)= 1$ for each $v\in V\setminus\{u_0,v_0\}$, there are exactly $2L$ edges between $u_0$ and $v_0$ with $L\in\mathbb{N}\cup\{0\}$. We call this graph $\mathcal{K}_{k,L}^{II}$ with $k\geq 1$; see Figure \ref{fig:K_kLII}.
	 \begin{figure}
		\begin{center}
			\includegraphics{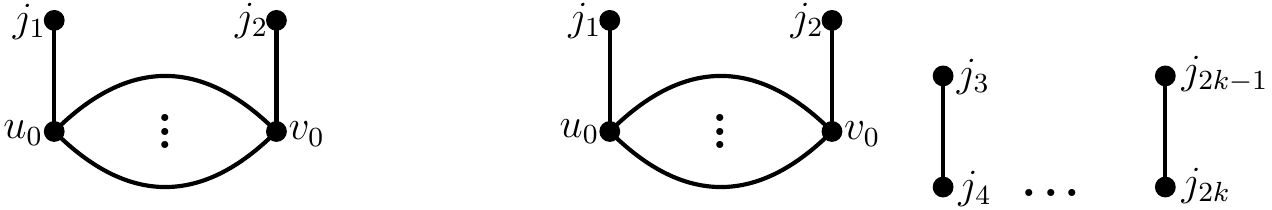}
			\caption{$\mathcal{K}_{1,L}^{II}$ left, $\mathcal{K}_{k,L}^{II}$ with $k\geq2$ right; each graph has $2L$ edges between $u_0$ and $v_0$}\label{fig:K_kLII}
		\end{center}
	\end{figure}
\end{enumerate}
To simplify notation, we write $\mathcal{H}_{k}$ for $\mathcal{H}_{k,0}$, $\mathcal{K}_{k}^I$ for $\mathcal{K}_{k,0}^I$, and $\mathcal{K}_{k}^{II}$ for $\mathcal{K}_{k,0}^{II}$. Because $\mathcal{H}_{k}$, $\mathcal{K}_{k}^I$ and $\mathcal{K}_{k}^{II}$ are the main characters in the rest of the section,  we illustrate them in Figures \ref{fig:H_p}, \ref{fig:K_pI} and~\ref{fig:K_pII} respectively. 
\begin{figure}
	\begin{center}
		\includegraphics{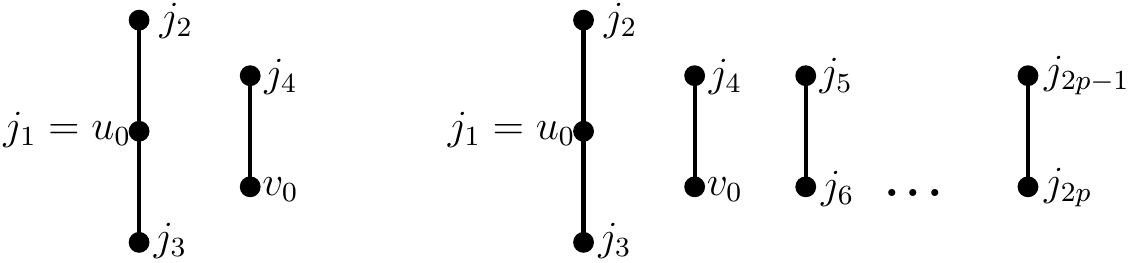}
		\caption{$\mathcal{H}_2$ left, $\mathcal{H}_p$ with $p\geq3$ right}\label{fig:H_p}
	\end{center}
\end{figure}
	 	\begin{figure}
	\begin{center}
		\includegraphics{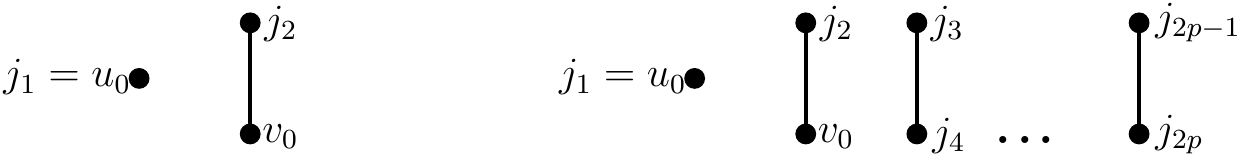}
		\caption{$\mathcal{K}_1^I$ left, $\mathcal{K}_p^I$ with $p\geq2$ right}\label{fig:K_pI}
	\end{center}
\end{figure}
	 	\begin{figure}
	\begin{center}
		\includegraphics{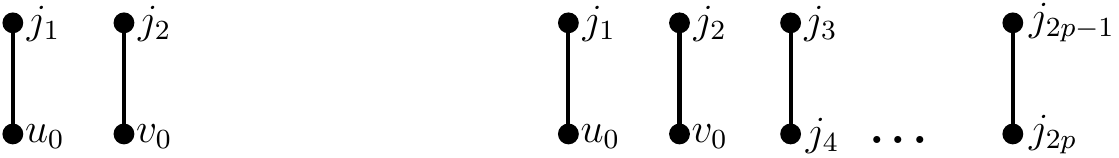}
		\caption{$\mathcal{K}_1^{II}$ left, $\mathcal{K}_p^{II}$ with $p\geq2$ right}\label{fig:K_pII}
	\end{center}
\end{figure}

The following lemma relate $\mathcal{H}_{k,L}$ to $\mathcal{H}_{k}$, $\mathcal{K}_{k,L}^I$ to $\mathcal{K}_{k}^I$, and $\mathcal{K}_{k,L}^{II}$ to $\mathcal{K}_{k}^{II}$.

\begin{lemma}\label{lem:red}
	Let $\mathcal{G}_{k,L}$ be one of the three graphs: $\mathcal{H}_{k,L}$ with $k\geq 2$, $\mathcal{K}_{k,L}^I$ with $k\in\mathbb{N}$, and $\mathcal{K}_{k,L}^{II}$ with $k\in\mathbb{N}$, $L\in\mathbb{N}\cup\{0\}$. Then there exists a nonnegative integer $N(L)$ such that
	\begin{equation}\label{eq:red1}
		R(\mathcal{G}_{k,L})=N(L)R(\mathcal{G}_{k}).
	\end{equation}
For any pairs of edges in $\mathcal{G}_{k,L}$: $(e_{k_1},e_{l_1}),\dots,(e_{k_n},e_{l_n})$ with $n\in\mathbb{N}$ satisfying that none of the $2n$ edges is an edge between $u_0$ and $v_0$, we have
\begin{equation}\label{eq:red2}
	R(\mathcal{G}_{k,L};\{e_{k_1},e_{l_1}\},\dots,\{e_{k_n},e_{l_n}\})=N(L)R(\mathcal{G}_k;\{e_{k_1},e_{l_1}\},\dots,\{e_{k_n},e_{l_n}\}).
\end{equation}
\end{lemma}
\begin{remark}
	The actual value of $N(L)$ is not important. We only need the fact that $N(L)\geq0$.
\end{remark}
\begin{proof}
	We only consider $\mathcal{G}_{k,L}=\mathcal{H}_{k,L}$ with $k\geq 2$ since the proofs for $\mathcal{K}_{k,L}^I$ and $\mathcal{K}_{k,L}^{II}$  are similar. In a partition $\mathscr{T}$ of $\mathcal{H}_{k,L}$, the restriction $u_0\centernot{\longleftrightarrow}v_0$ in $(V,\mathcal{E}_{n(\mathscr{T})}\cup\mathcal{E}_{n(\mathscr{T})+1})$ implies that none of the $2L$ edges between $u_0$ and $v_0$ is in $\mathcal{E}_{n(\mathscr{T})}$ and $\mathcal{E}_{n(\mathscr{T})+1}$; $v_0\notin\{j_1,\dots,j_{2k}\}$ implies that $j_4\in P_{n(\mathscr{T})}$ and $j_4v_0\in\mathcal{E}_{n(\mathscr{T})}$. 
	So in order to get a partition of $\mathcal{H}_{k,L}$, we may first ignore all the $2L$ edges between $u_0$ and $v_0$. That is, we start with a partition of $\mathcal{H}_{k}$: $\mathcal{E}_1,\dots,\mathcal{E}_{k+1}$. Then we add those $2L$ edges into the $k-1$ sets of edges $\mathcal{E}_1,\dots,\mathcal{E}_{n(\mathscr{T})-1},\mathcal{E}_{n(\mathscr{T})+2},\dots,\mathcal{E}_{k+1}$ such that each edge between $u_0$ and $v_0$ pairs with another such edge and this pair can be put in any one of the $k-1$ sets of edges. Therefore, $N(L)$ is the number of ways to distribute $2L$ distinct objects in $k-1$ distinct boxes such that each box having even number (could be $0$) of objects.
	
	The proof of \eqref{eq:red2} is similar to that of \eqref{eq:red1}.
\end{proof}

Lemma \ref{lem:red} implies that the proof of Proposition \ref{prop:stronger} is complete, modulo the proof of the following proposition.
\end{proof}
\begin{proposition}\label{prop:RGsgn}
	Let $\mathcal{G}=(V,\mathcal{E})$ be one of the three graphs: $\mathcal{H}_{k}$ with $k\geq 2$, $\mathcal{K}_{k}^I$ with $k\in\mathbb{N}$ and $\mathcal{K}_{k}^{II}$ with $k\in\mathbb{N}$. Then for any sequence of pairs of edges in $\mathcal{E}$,
	\begin{equation}
		(e_{k_1},e_{l_1}),\dots,(e_{k_n},e_{l_n}), n\in\mathbb{N},
	\end{equation}
We have
\begin{align}
	&(-1)^{k-1}R(\mathcal{G})\geq0,\label{eq:RGsgn1}\\
	&(-1)^{k-1}R(\mathcal{G};\{e_{k_1},e_{l_1}\},\dots,\{e_{k_n},e_{l_n}\})\geq0.\label{eq:RGmsgn1}
\end{align}
\end{proposition}

\subsection{Proof for the special graphs}
In this subsection, we prove Proposition \ref{prop:RGsgn} and thus Propositions \ref{prop:stronger} and \ref{prop:Rmsgn}.
\begin{proof}[Proof of Proposition \ref{prop:RGsgn}]
Note that, a priori, we should fix both $\{j_1,\dots,j_{2k}\}\subseteq V$ and $u_0v_0\in E$ as we did before this proposition. But here and only in this proof, we allow $k$ to take any integer value. To emphasize this, in the rest of the proof, we usually use $p$ which plays the role of $k$. 
 
We first prove a lemma which will be used to determine $R(\mathcal{H}_p)$, $R(\mathcal{K}_p^I)$ and $R(\mathcal{K}_p^{II})$.
 \begin{lemma}\label{lem:parsum}
 	For $p\in\mathbb{N}$, we have
 	\begin{equation}
 		I_p:=\sum_{\mathscr{P}}(-1)^{|\mathscr{P}|-1}(|\mathscr{P}|-1)!=\begin{cases}
 			1,&p=1\\
 			0,&p\geq 2,
 		\end{cases}
 	\end{equation}
 where the sum is over all partitions $\mathscr{P}$ of $\{j_1,\dots,j_{2p}\}$ satisfying that $j_{2q-1}$ and $j_{2q}$ are in the same block for each $1\leq q\leq p$. 
 \end{lemma}
\begin{remark}
	Lemma \ref{lem:parsum} can be viewed as a statement about partitions of a set of $p$ elements, where one identifies $j_{2q-1}$ and $j_{2q}$ as the same element.
\end{remark}
\begin{proof}
	This is already proved in \cite{Shl86}. For completeness, we give another proof here. Let $\hat{\mathcal{G}}_p$ be the graph consisting of $p$ trees each of which has two vertices; see Figure \ref{fig:tildeG_p}. We may define an Ising model on $\hat{\mathcal{G}}_p$ by setting $J_{j_{2q-1}j_{2q}}=\beta>0$ for each $1\leq q\leq p$.
\begin{figure}
	\begin{center}
		\includegraphics{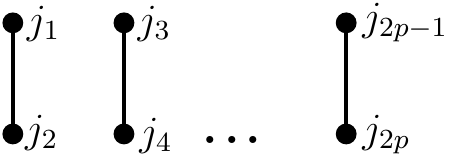}
		\caption{$\hat{\mathcal{G}_p}$ with $p\in\mathbb{N}$}\label{fig:tildeG_p}
	\end{center}
\end{figure}
	Then by~\eqref{eq:ursdef2}, we have
	\begin{align}
		u_{2p}^{\hat{\mathcal{G}}_p}(\sigma_{j_1},\dots,\sigma_{j_{2p}})&=\sum_{\mathscr{P}}(-1)^{|\mathscr{P}|-1}(|\mathscr{P}|-1)!\prod_{P\in\mathscr{P}}\langle \sigma_P\rangle_{\hat{\mathcal{G}}_p}\nonumber\\
		&=\prod_{q=1}^p\langle \sigma_{j_{2q-1}}\sigma_{j_{2q}}\rangle_{\hat{\mathcal{G}}_p}\sum_{\mathscr{P}}(-1)^{|\mathscr{P}|-1}(|\mathscr{P}|-1)!,
	\end{align}
where the first sum should be over all partitions $\mathscr{P}$ of $\{j_1,\dots,j_{2p}\}$, but by the special structure of $\hat{\mathcal{G}}_p$, only the partitions described in the lemma make nonzero contributions; this is because spins from different connected components of $\hat{\mathcal{G}}_p$ are independent and $\langle \sigma_{j_l}\rangle_{\hat{\mathcal{G}}_p}=0$ for any $l=1,\dots,2p$. From definition \eqref{eq:ursdef1}, we know
\begin{equation}
	u_{2p}^{\hat{\mathcal{G}}_p}(\sigma_{j_1},\dots,\sigma_{j_{2p}})=0 \text{ if } j_1,\dots, j_{2p} \text{ are not connected in }\hat{\mathcal{G}}_p.
\end{equation}
This completes the proof of the lemma.
\end{proof}

Our next lemma gives the values of $R(\mathcal{H}_p)$, $R(\mathcal{K}_p^I)$ and $R(\mathcal{K}_p^{II})$.
\begin{lemma}\label{lem:Gpsgn}
	\begin{equation}
		R(\mathcal{K}_p^{I})=R(\mathcal{K}_p^{II})=\begin{cases}
			1, &p=1\\
			0, &p\geq 2,
		\end{cases}
	R(\mathcal{H}_p)=\begin{cases}
		-2,&p=2\\
		0, &p\geq 3.
	\end{cases}
	\end{equation}
\end{lemma}
\begin{proof}
	$R(\mathcal{K}_1^{I})=1$ is easily checked by definition. Pick $\bf{m}\in\mathbb{N}_0^E$ such that $\mathcal{K}_p^{I}$ is the graph associated with $\bf{m}$. By Lemma \ref{lem:Rmtod}, we have
	\begin{equation}\label{eq:RKd}
		R(\mathcal{K}_p^{I})=R(\bf{m})=\left.\frac{\partial^{\bf{m}_{e_1}+\dots+\bf{m}_{e_{|E|}}}}{\partial J^{\bf{m}_{e_1}}_{e_1}\dots\partial J^{\bf{m}_{e_{|E|}}}_{e_{|E|}}}\left[\left(\frac{Z_G}{2^{|V|}}\right)^{p+1}\frac{\partial u_{2p}(\sigma_{j_1},\dots,\sigma_{j_{2p}})}{\partial J_{u_0v_0}}\right]\right|_{\bf{J}=\bf{0}}.
	\end{equation}
From \eqref{eq:ursdef2}, we have
\begin{align}
	&\frac{\partial u_{2p}(\sigma_{j_1},\dots,\sigma_{j_{2p}})}{\partial J_{u_0v_0}}\nonumber\\
	&\quad=\sum_{\mathscr{P}}(-1)^{|\mathscr{P}|-1}(|\mathscr{P}|-1)!\sum_{Q\in\mathscr{P}}\left[\langle \sigma_Q\sigma_{u_0}\sigma_{v_0}\rangle_G-\langle\sigma_Q\rangle_G\langle\sigma_{u_0}\sigma_{v_0}\rangle_G\right]\prod_{\substack{P\in\mathscr{P}:\\P\neq Q}}\langle\sigma_P\rangle_G,
\end{align}
where the sum is over all partitions $\mathscr{P}$ of $\{j_1,\dots,j_{2p}\}$ such that each block $Q\in\mathscr{P}$ satisfies $|Q|\in 2\mathbb{N}$. When we compute $R(\bf{m})$ using \eqref{eq:RKd}, this formula suggests that before taking the partial derivatives outside the brackets, we may set $J_e=0$ for those $e\in E$ satisfying $\bf{m}_e=0$; this is equivalent to  changing the graph $G=(V,E)$ where our Ising model is defined to  $\hat{G}=(\hat{V},\hat{E})$ with
\begin{equation}
	\hat{V}:=V, \hat{E}:=\{e\in E: \bf{m}_e>0\}.
\end{equation}
Then $\hat{G}$ is exactly $\mathcal{K}_p^I$ plus possibly some isolated vertices which are irrelevant. Since $u_0$ is isolated in $\mathcal{K}_p^I$, we have
\begin{align}\label{eq:uK}
	\left.\frac{\partial u_{2p}(\sigma_{j_1},\dots,\sigma_{j_{2p}})}{\partial J_{u_0v_0}}\right|_{J_e=0,e\in E\setminus\hat{E}}&=\sum_{\mathscr{P}}(-1)^{|\mathscr{P}|-1}(|\mathscr{P}|-1)!\sum_{Q\in\mathscr{P}}\langle \sigma_Q\sigma_{u_0}\sigma_{v_0}\rangle_{\mathcal{K}_p^I}\prod_{\substack{P\in\mathscr{P}:\\P\neq Q}}\langle\sigma_P\rangle_{\mathcal{K}_p^I}\nonumber\\
	&=\langle\sigma_{j_2}\sigma_{v_0}\rangle_{\mathcal{K}_p^I}\prod_{q=2}^p\langle\sigma_{j_{2q-1}}\sigma_{j_{2q}}\rangle_{\mathcal{K}_p^I}\sum_{\mathscr{P}}(-1)^{|\mathscr{P}|-1}(|\mathscr{P}|-1)!,
\end{align}
where the last sum is over all partitions $\mathscr{P}$ of $\{j_1,\dots,j_{2p}\}$ satisfying that $j_{2q-1}$ and $j_{2q}$ are in the same block for each $1\leq q\leq p$.  This, combined with Lemma \ref{lem:parsum}, proves the formula for $R(\mathcal{K}_p^{I})$.

$R(\mathcal{K}_1^{II})=1$ is also easily checked by definition. The proof of the formula for general $R(\mathcal{K}_p^{II})$ is similar to that of $R(\mathcal{K}_p^{I})$.

Finally, we prove the formula for $R(\mathcal{H}_p)$. For $\mathcal{H}_2$, all possible partitions are
\begin{enumerate}
	\item $\mathcal{E}_1=\{j_1j_2,j_1j_3,j_4v_0\}$, $\mathcal{E}_2=\mathcal{E}_3=\emptyset$,
	\item $\mathcal{E}_1=\{j_1j_2\}$, $\mathcal{E}_2=\{j_1j_3,j_4v_0\}$, $\mathcal{E}_3=\emptyset$,
	\item $\mathcal{E}_1=\{j_1j_3\}$, $\mathcal{E}_2=\{j_1j_2,j_4v_0\}$, $\mathcal{E}_3=\emptyset$,
	\item $\mathcal{E}_1=\{j_1j_2,j_1j_3\}$, $\mathcal{E}_2=\{j_4v_0\}$, $\mathcal{E}_3=\emptyset$,
\end{enumerate}
which gives $R(\mathcal{H}_2)=-2$.
The corresponding formula of \eqref{eq:uK} for $\mathcal{H}_p$ case is
\begin{equation}\label{eq:uG}
	\left.\frac{\partial u_{2p}(\sigma_{j_1},\dots,\sigma_{j_{2p}})}{\partial J_{u_0v_0}}\right|_{J_e=0,e\in E\setminus\hat{E}}=\sum_{\mathscr{P}}(-1)^{|\mathscr{P}|-1}(|\mathscr{P}|-1)!\sum_{Q\in\mathscr{P}}\langle \sigma_Q\sigma_{u_0}\sigma_{v_0}\rangle_{\mathcal{H}_p}\prod_{\substack{P\in\mathscr{P}:\\P\neq Q}}\langle\sigma_P\rangle_{\mathcal{H}_p},
\end{equation}
where of course $\hat{E}$ should be modified according to $\mathcal{H}_p$; the sum over $\mathscr{P}$ can be divided into the following four cases.
\begin{enumerate}[i.]
	\item $\{j_1,j_2,j_3,j_4\}\subseteq Q$. The contribution to \eqref{eq:uG} is
	\begin{equation}
		\langle\sigma_{j_2}\sigma_{j_3}\rangle_{\mathcal{H}_p}\langle\sigma_{j_4}\sigma_{v_0}\rangle_{\mathcal{H}_p}\prod_{q=3}^{p}\langle\sigma_{j_{2q-1}}\sigma_{j_{2q}}\rangle_{\mathcal{H}_p}\sum_{\mathscr{P}}(-1)^{|\mathscr{P}|-1}(|\mathscr{P}|-1)!,
	\end{equation}
	where the sum is over all partitions $\mathscr{P}$ of $\{j_1,\dots,j_{2p}\}$ satisfying that $j_1, j_2,j_3,j_4$ are in the same block, $j_{2q-1}$ and $j_{2q}$ are in the same block for each $3\leq q\leq p$;  in the notation of Lemma \ref{lem:parsum}, this sum is equal to
	\begin{equation}
		I_{p-1}=\begin{cases}
			1,&p=2\\
			0,&p\geq 3.
		\end{cases}
	\end{equation}

	\item $\{j_1,j_4\}\subseteq Q$, $j_2\notin Q$, $j_3\notin Q$. The contribution to \eqref{eq:uG} is
	\begin{equation}
		\langle\sigma_{j_2}\sigma_{j_3}\rangle_{\mathcal{H}_p}\langle\sigma_{j_4}\sigma_{v_0}\rangle_{\mathcal{H}_p}\prod_{q=3}^{p}\langle\sigma_{j_{2q-1}}\sigma_{j_{2q}}\rangle_{\mathcal{H}_p}\sum_{\mathscr{P}}(-1)^{|\mathscr{P}|-1}(|\mathscr{P}|-1)!,
	\end{equation}
where the sum is over all partitions $\mathscr{P}$ of $\{j_1,\dots,j_{2p}\}$ satisfying that $j_1$ and $j_4$, $j_2$ and $j_3$, $j_{2q-1}$ and $j_{2q}$ are in the same block for each $3\leq q\leq p$, and $j_1,j_2,j_3,j_4$ are not in the same block; in the notation of Lemma \ref{lem:parsum}, this sum is equal to
\begin{equation}
	I_p-I_{p-1}=\begin{cases}
		-1,&p=2\\
		0,&p\geq 3.
	\end{cases}
\end{equation}

\item $\{j_2,j_4\}\subseteq Q$, $j_1\notin Q$, $j_3\notin Q$. The contribution to \eqref{eq:uG} is
\begin{equation}
	\langle\sigma_{j_1}\sigma_{j_2}\rangle_{\mathcal{H}_p}\langle\sigma_{j_1}\sigma_{j_3}\rangle_{\mathcal{H}_p}\langle\sigma_{j_4}\sigma_{v_0}\rangle_{\mathcal{H}_p}\prod_{q=3}^{p}\langle\sigma_{j_{2q-1}}\sigma_{j_{2q}}\rangle_{\mathcal{H}_p}\sum_{\mathscr{P}}(-1)^{|\mathscr{P}|-1}(|\mathscr{P}|-1)!,
\end{equation}
where the sum is over all partitions $\mathscr{P}$ of $\{j_1,\dots,j_{2p}\}$ satisfying that $j_2$ and $j_4$, $j_1$ and $j_3$, $j_{2q-1}$ and $j_{2q}$ are in the same block for each $3\leq q\leq p$, and $j_1,j_2,j_3,j_4$ are not in the same block; in the notation of Lemma \ref{lem:parsum}, this sum is equal to
\begin{equation}
	I_p-I_{p-1}=\begin{cases}
		-1,&p=2\\
		0,&p\geq 3.
	\end{cases}
\end{equation}

\item $\{j_3,j_4\}\subseteq Q$, $j_1\notin Q$, $j_2\notin Q$. The contribution to \eqref{eq:uG} is
\begin{equation}
	\langle\sigma_{j_1}\sigma_{j_2}\rangle_{\mathcal{H}_p}\langle\sigma_{j_1}\sigma_{j_3}\rangle_{\mathcal{H}_p}\langle\sigma_{j_4}\sigma_{v_0}\rangle_{\mathcal{H}_p}\prod_{q=3}^{p}\langle\sigma_{j_{2q-1}}\sigma_{j_{2q}}\rangle_{\mathcal{H}_p}\sum_{\mathscr{P}}(-1)^{|\mathscr{P}|-1}(|\mathscr{P}|-1)!,
\end{equation}
where the sum is over all partitions $\mathscr{P}$ of $\{j_1,\dots,j_{2p}\}$ satisfying that $j_{2q-1}$ and $j_{2q}$ are in the same block for each $1\leq q\leq p$, and $j_1,j_2,j_3,j_4$ are not in the same block; in the notation of Lemma \ref{lem:parsum}, this sum is equal to
\begin{equation}
	I_p-I_{p-1}=\begin{cases}
		-1,&p=2\\
		0,&p\geq 3.
	\end{cases}
\end{equation}
\end{enumerate}
This completes the proof of the formula for $R(\mathcal{H}_p)$ when $p\geq 3$, and thus the lemma.
\end{proof}
	 
Lemma \ref{lem:Gpsgn} finishes the proof of \eqref{eq:RGsgn1}.

We now complete the proof Proposition \ref{prop:RGsgn} by proving  \eqref{eq:RGmsgn1} for $\mathcal{G}_p:=\mathcal{K}_p^I$, $\mathcal{K}_p^{II}$ and $\mathcal{H}_p$, which we state as the following lemma.
\begin{lemma}\label{lem:RGmsgn2}
Let $\mathcal{G}_p$ be one of the three graphs: $\mathcal{K}_p^I$ with $p\in\mathbb{N}$, $\mathcal{K}_p^{II}$ with $p\in\mathbb{N}$ and $\mathcal{H}_p$ with $p\geq 2$. Then for any pairs of edges in $\mathcal{G}_p$: $(e_{k_1},e_{l_1}),\dots,(e_{k_n},e_{l_n})$ with $n\in\mathbb{N}$,
\begin{equation}\label{eq:RGmsgn2}
	(-1)^{p-1}R(\mathcal{G}_p;\{e_{k_1},e_{l_1}\},\dots,\{e_{k_n},e_{l_n}\})\geq0.
\end{equation}
\end{lemma}
\begin{proof}
We prove \eqref{eq:RGmsgn2} by induction on $p\in\mathbb{N}$ for $\mathcal{G}_p:=\mathcal{K}_p^I$ and $\mathcal{K}_p^{II}$, and a separate argument for $\mathcal{G}_p=\mathcal{H}_p$. We may assume that $e_{k_q}\neq e_{l_q}$ for any $1\leq q\leq n$ since otherwise the LHS of \eqref{eq:RGmsgn2} is trivially $0$. 

If $p=1$,  the LHS of \eqref{eq:RGmsgn2} is $0$ for $\mathcal{G}_1=\mathcal{K}_1^I$ since $\mathcal{K}_1^I$ only contains one edge; it is also $0$ for $\mathcal{G}_1=\mathcal{K}_1^{II}$ since $j_1u_0$ and $j_2v_0$ should be in the same subgraph of the unique partition of $\mathcal{K}_1^{II}$.

We suppose \eqref{eq:RGmsgn2} holds for any $\mathcal{G}_p:=\mathcal{K}_p^I$ and $\mathcal{K}_p^{II}$ with $p\leq N\in\mathbb{N}$ and any $n$ ($\in\mathbb{N}$) pairs of edges $(e_{k_1},e_{l_1}),\dots,(e_{k_n},e_{l_n})$ in $\mathcal{G}_p$.

We next consider $\mathcal{G}_{N+1}$ and $n$ ($\in\mathbb{N}$) pairs of edges $(e_{k_1},e_{l_1}),\dots,(e_{k_n},e_{l_n})$ in $\mathcal{G}_{N+1}$. We have
\begin{align}\label{eq:RHN+1}
	&R(\mathcal{G}_{N+1})=R(\mathcal{G}_{N+1};[e_{k_1},e_{l_1}])+R(\mathcal{G}_{N+1};\{e_{k_1},e_{l_1}\})\nonumber\\
	&\quad=R(\mathcal{G}_{N+1};[e_{k_1},e_{l_1}])+R(\mathcal{G}_{N+1};\{e_{k_1},e_{l_1}\},[e_{k_2},e_{l_2}])+R(\mathcal{G}_{N+1};\{e_{k_1},e_{l_1}\},\{e_{k_2},e_{l_2}\})\nonumber\\
	&\quad=\dots\nonumber\\
	&\quad=R(\mathcal{G}_{N+1};[e_{k_1},e_{l_1}])+R(\mathcal{G}_{N+1};\{e_{k_1},e_{l_1}\},[e_{k_2},e_{l_2}])+\dots+\nonumber\\
	&\qquad R(\mathcal{G}_{N+1};\{e_{k_1},e_{l_1}\},\dots,[e_{k_n},e_{l_n}])+R(\mathcal{G}_{N+1};\{e_{k_1},e_{l_1}\},\dots,\{e_{k_n},e_{l_n}\})
\end{align}

By Lemma \ref{lem:Gpsgn}, we know $(-1)^NR(\mathcal{G}_{N+1})\geq0$. So our proof by induction for \eqref{eq:RGmsgn2} when $\mathcal{G}_p:=\mathcal{K}_p^I$ and $\mathcal{K}_p^{II}$ is complete if we can prove
\begin{claim} For $\mathcal{G}_p=\mathcal{K}_p^I$ and $\mathcal{K}_p^{II}$, we have
	\begin{equation}\label{eq:Hsgn}
	(-1)^{N}R(\mathcal{G}_{N+1};\{e_{k_1},e_{l_1}\},\dots,\{e_{k_{q-1}},e_{l_{q-1}}\},[e_{k_q},e_{l_q}])\leq 0, \text{ for each }1\leq q\leq n.
\end{equation}
\end{claim}
\begin{remark}
	This claim does not hold for all $\mathcal{H}_p$, and that is why we need a different argument for the case $\mathcal{G}_p=\mathcal{H}_p$. Here is a counterexample. For $\mathcal{H}_2$, the only partition satisfies the restrictions $\{j_1j_2,j_1j_3\}$ and $[j_1j_2,j_4v_0]$ is $\mathcal{E}_1=\{j_1j_3\}$ and $\mathcal{E}_2=\{j_1j_2,j_4v_0\}$; so we have
	\begin{equation}
		(-1)^{2-1}R(\mathcal{H}_2;\{j_1j_2,j_1j_3\},[j_1j_2,j_4v_0])=1>0.
	\end{equation}
\end{remark}
\begin{proof}
	There are two cases.
	\begin{enumerate}[I.]
		\item $\mathcal{G}_{N+1}=\mathcal{K}_{N+1}^I$. $e_{k_q}\neq e_{l_q}$ implies that at least one of $e_{k_q}$ and $e_{l_q}$ has both endpoints in $\{j_1,\dots,j_{2p}\}$; we may just delete this edge (with both endpoints in $\{j_1,\dots,j_{2p}\}$) from $\mathcal{K}_{N+1}^I$ to get a $\mathcal{K}_{N}^I$ with a possible relabeling of vertices other than $\{j_1,j_2,u_0,v_0\}$. Then we have
		\begin{equation}
			R(\mathcal{K}_{N+1}^I;\{e_{k_1},e_{l_1}\},\dots,\{e_{k_{q-1}},e_{l_{q-1}}\},[e_{k_q},e_{l_q}])=R(\mathcal{K}_{N}^I;\{\tilde{e}_{k_1},\tilde{e}_{l_1}\},\dots,\{\tilde{e}_{k_{q-1}},\tilde{e}_{l_{q-1}}\}),
		\end{equation}
	where the tildes indicate that we may need to change the list of pairs of edges due to the relabeling of vertices. By the induction hypothesis,
	\begin{equation}
		(-1)^{N-1}R(\mathcal{K}_{N}^I;\{\tilde{e}_{k_1},\tilde{e}_{l_1}\},\dots,\{\tilde{e}_{k_{q-1}},\tilde{e}_{l_{q-1}}\})\geq0.
	\end{equation}
	This completes the proof of the claim for the case $\mathcal{G}_{N+1}=\mathcal{K}_{N+1}^I$.
	
	\item $\mathcal{G}_{N+1}=\mathcal{K}_{N+1}^{II}$. Note that for any partition of $\mathcal{K}_{N+1}^{II}$, $j_1u_0$ and $j_2v_0$ should be in the same subgraph. We may assume that as an unordered pair $\{e_{k_q},e_{l_q}\}\neq\{j_1u_0,j_2v_0\}$ since otherwise the LHS of \eqref{eq:RGmsgn2} is just 0. The rest of the proof is similar to that of Case~I.
	\end{enumerate}
This completes the proof of the claim.
\end{proof}

Finally, we are left with the proof of \eqref{eq:RGmsgn2} for the case $\mathcal{G}_p=\mathcal{H}_p$ with $p\geq 2$. Recall that in the definition of a partition $\mathscr{T}$ of $\mathcal{H}_p$, we denote by $P_1,\dots,P_{n(\mathscr{T})}$ a partition of $\{j_1,\dots,j_{2p}\}$. We have
\begin{align}\label{eq:Hdec}
	R(\mathcal{H}_p;\{e_{k_1},e_{l_1}\},\dots,\{e_{k_n},e_{l_n}\})=&R(\mathcal{H}_p;\{e_{k_1},e_{l_1}\},\dots,\{e_{k_n},e_{l_n}\},j_1\in P_{n(\mathscr{T})})+\nonumber\\
	&R(\mathcal{H}_p;\{e_{k_1},e_{l_1}\},\dots,\{e_{k_n},e_{l_n}\}, j_1\notin P_{n(\mathscr{T})})\nonumber\\
	=&R(\mathcal{H}_p;\{e_{k_1},e_{l_1}\},\dots,\{e_{k_n},e_{l_n}\},j_1\in P_{n(\mathscr{T})})+\nonumber\\
	&2R(\mathcal{H}_p;\{e_{k_1},e_{l_1}\},\dots,\{e_{k_n},e_{l_n}\}, j_2\in P_{n(\mathscr{T})},j_3\notin P_{n(\mathscr{T})}),
\end{align}
where $R(\mathcal{H}_p;\dots,j_1\in P_{n(\mathscr{T})})$ is defined as $R(\mathcal{H}_p;\dots)$ with the extra restriction that the partitions of $\mathcal{H}_p$ satisfying $j_1\in P_{n(\mathscr{T})}$; the second equality follows since $j_1\notin P_{n(\mathscr{T})}$ implies that exactly one of $j_2$ and $j_3$ is in $P_{n(\mathscr{T})}$ and by symmetry
\begin{equation}
	R(\mathcal{H}_p;\dots, j_2\in P_{n(\mathscr{T})},j_3\notin P_{n(\mathscr{T})})=R(\mathcal{H}_p;\dots, j_2\notin P_{n(\mathscr{T})},j_3\in P_{n(\mathscr{T})}).
\end{equation}

Note that $j_1\in P_{n(\mathscr{T})}$ implies that $j_1$ and $j_4$ should be in $P_{n(\mathscr{T})}$, and $j_2$ and $j_3$ should be in the same $P_l$ for some $1\leq l\leq n(\mathscr{T})$. So the transformation illustrated in Figure \ref{fig:H_ptoK_pI} implies that
\begin{figure}
	\begin{center}
		\includegraphics{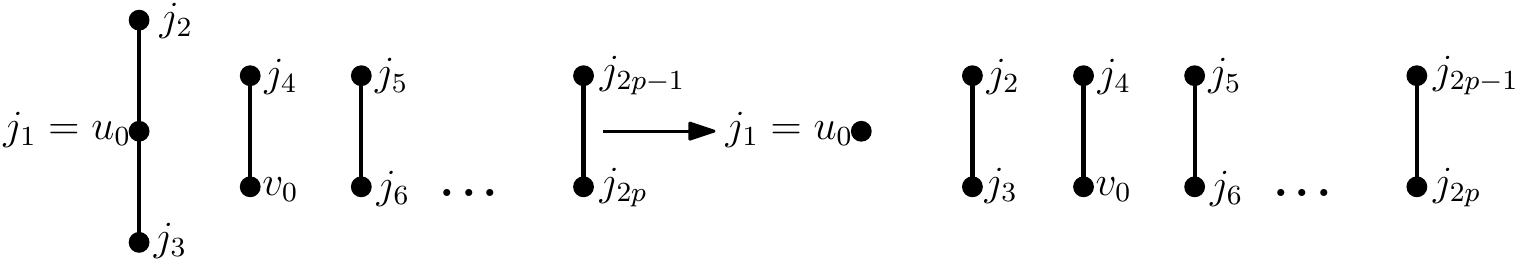}
		\caption{The transformation from $\mathcal{H}_p$ with $j_1\in P_{n(\mathscr{T})}$ to $\tilde{\mathcal{K}}_p^I$, tilde means a relabeling of $\mathcal{K}_p^I$}\label{fig:H_ptoK_pI}
	\end{center}
\end{figure}
\begin{equation}\label{eq:Hdec1}
	R(\mathcal{H}_p;\{e_{k_1},e_{l_1}\},\dots,\{e_{k_n},e_{l_n}\}, j_1\in P_{n(\mathscr{T})})=R(\tilde{\mathcal{K}}_p^I;\{e_{k_1},e_{l_1}\},\dots,\{e_{k_n},e_{l_n}\}).
\end{equation}
Note that $j_2\in P_{n(\mathscr{T})}$ and $j_3\notin P_{n(\mathscr{T})}$ imply that $j_2$ and $j_4$ should be in $P_{n(\mathscr{T})}$, and $j_1$ and $j_3$ should be in some $1\leq l\leq n(\mathscr{T})-1$. So the transformation illustrated in Figure \ref{fig:H_ptoK_pII} implies that
\begin{figure}
	\begin{center}
		\includegraphics{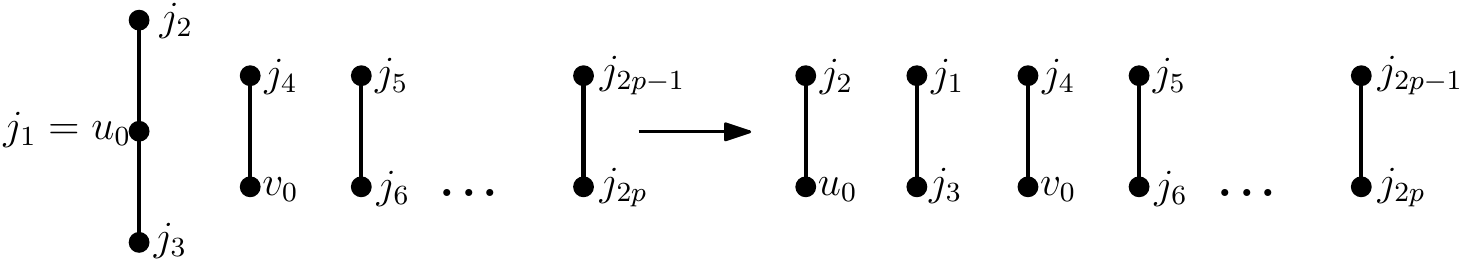}
		\caption{The transformation from $\mathcal{H}_p$ with $j_2\in P_{n(\mathscr{T})}, j_3\notin P_{n(\mathscr{T})}$ to $\tilde{\mathcal{K}}_p^{II}$ with $\{j_1j_3,j_2u_0\}$, tilde means a relabeling of $\mathcal{K}_p^{II}$}\label{fig:H_ptoK_pII}
	\end{center}
\end{figure}
\begin{align}\label{eq:Hdec2}
	&R(\mathcal{H}_p;\{e_{k_1},e_{l_1}\},\dots,\{e_{k_n},e_{l_n}\}, j_2\in P_{n(\mathscr{T})},j_3\notin P_{n(\mathscr{T})})\nonumber\\
	&\quad=R(\tilde{\mathcal{K}}_p^{II};\{e_{k_1},e_{l_1}\},\dots,\{e_{k_n},e_{l_n}\},\{j_1j_3,j_2u_0\}).
\end{align}

Therefore, by \eqref{eq:Hdec}, \eqref{eq:Hdec1}, \eqref{eq:Hdec2} and \eqref{eq:RGmsgn2} for $\mathcal{G}_p:=\mathcal{K}_p^I$ and $\mathcal{K}_p^{II}$, we have
\begin{equation}
	(-1)^{p-1}R(\mathcal{H}_p;\{e_{k_1},e_{l_1}\},\dots,\{e_{k_n},e_{l_n}\})\geq0, \forall p\geq 2.
\end{equation}
This completes the proof of \eqref{eq:RGmsgn2} for the case $\mathcal{G}_p=\mathcal{H}_p$ and thus the lemma.
\end{proof}
As mentioned right before Lemma \ref{lem:RGmsgn2}, this also completes the proof of Proposition \ref{prop:RGsgn}.
\end{proof}

\section{Proof of Theorem \ref{thm:umon}}\label{sec:proofmain}
In this section, we prove Theorem \ref{thm:umon}. We first prove an ancillary lemma.
\begin{lemma}\label{lem:1to2}
	 Suppose that for any finite graph $G=(V,E)$ and any ferromagnetic pair interactions $\bf{J}$ on $G$, any $k\in\mathbb{N}$, any $j_1,\dots,j_{2k}\in V$ ($j_l$'s are not necessarily distinct), any $u_0v_0\in E$ with $v_0\notin\{j_1,\dots,j_{2k}\}$, we have
	 \begin{equation}\label{eq:1to2}
	 	(-1)^{k-1}\frac{\partial u_{2k}(\sigma_{j_1},\dots,\sigma_{j_{2k}})}{\partial J_{u_0v_0}}\geq 0.
	 \end{equation}
 Then for any finite graph $G$ and any ferromagnetic pair interactions $\bf{J}$ on $G$, any $u_0v_0\in E$ with $u_0\in\{j_1,\dots,j_{2k}\}$ and $v_0\in\{j_1,\dots,j_{2k}\}$, we also have
 \begin{equation}
 	(-1)^{k-1}\frac{\partial u_{2k}(\sigma_{j_1},\dots,\sigma_{j_{2k}})}{\partial J_{u_0v_0}}\geq 0.
 \end{equation}
\end{lemma}
\begin{proof}
	We add a new vertex $w$ to $G$ to get a new graph $\hat{G}=(\hat{V},\hat{E})$ with
	\begin{equation}
		\hat{V}=V\cup\{w\}, \hat{E}=\left(E\setminus\{u_0v_0\}\right)\cup\{u_0w,v_0w\}.
	\end{equation}
The ferromagnetic interactions $\hat{\bf{J}}$ on $\hat{G}$ are defined by
\begin{equation}
	\hat{J}_{uv}:=\begin{cases}
		J_{uv}, &uv\notin\{u_0w,v_0w\}\\
		\hat{\beta},&uv\in\{u_0w,v_0w\}.
	\end{cases}
\end{equation}
We set $J_{u_0v_0}:=\beta\geq0$ and
\begin{equation}
	\cosh(2\hat{\beta})=\exp(2\beta).
\end{equation}
Then a direct computation gives
\begin{equation}
	\langle\sigma_A\rangle_G=\langle\sigma_A\rangle_{\hat{G}}, \forall A\subseteq V.
\end{equation}
So we have
\begin{align}
	\left.\frac{\partial u_{2k}^G(\sigma_{j_1},\dots,\sigma_{j_{2k}})}{\partial J_{u_0v_0}}\right|_{J_{u_0v_0}=\beta}=&\left.\frac{\partial u_{2k}^{\hat{G}}(\sigma_{j_1},\dots,\sigma_{j_{2k}})}{\partial \hat{J}_{u_0w}}\right|_{\hat{J}_{u_0w}=\hat{J}_{v_0w}=\hat{\beta}}\frac{d\hat{\beta}}{d\beta}+\nonumber\\
	&\left.\frac{\partial u_{2k}^{\hat{G}}(\sigma_{j_1},\dots,\sigma_{j_{2k}})}{\partial \hat{J}_{v_0w}}\right|_{\hat{J}_{u_0w}=\hat{J}_{v_0w}=\hat{\beta}}\frac{d\hat{\beta}}{d\beta},
\end{align}
where on the RHS of the equality, of course we have $\hat{J}_{uv}=J_{uv}$ for each $uv\in E\setminus\left\{u_0v_0\right\}$. The lemma follows from \eqref{eq:1to2} for $\hat{G}$ and $\hat{\bf{J}}$, and the simple observation that $d\hat{\beta}/d\beta>0$.
\end{proof}

The following reduction formula for Ursell functions (see (3.10) of \cite{Syl75}, or (13) on p.~55 of \cite{Syl76}) is useful when some of the arguments are the same.
\begin{lemma}[(3.10) of \cite{Syl75}]\label{lem:reduction}
	For $k\in\mathbb{N}$, let $\{\sigma_{j_1},\dots,\sigma_{j_{2k}}\}$ be a family of $2k$ random variables (not necessarily distinct). If $j_1=j_2$ and $k\geq 2$, then
	\begin{equation}\label{eq:reduction}
		u_{2k}(\sigma_{j_1},\dots,\sigma_{j_{2k}})=-\sum_{\substack{A: 1\in A, 2\notin A\\A\subseteq\{1,\dots,2k\}}}u_{|A|}(\sigma_{j_A})u_{2k-|A|}(\sigma_{j_{A^c}}), 
		\end{equation}
where 
\begin{equation}
	u_{|A|}(\sigma_{j_A}):=u_q(\sigma_{j_{l_1}},\dots,\sigma_{j_{l_q}}) \text{ if }A=\{l_1,\dots,l_q\}, A^c:=\{1,\dots,2k\}\setminus A.
\end{equation}
\end{lemma}
\begin{proof}
	See pp. 66-71 of \cite{Syl76}.
\end{proof}
We remark that, in general, \eqref{eq:reduction} does not hold for $k=1$. We are ready to prove Theorem~\ref{thm:umon}.
\begin{proof}[Proof of Theorem \ref{thm:umon}]
	By Lemma \ref{lem:1to2}, it is enough to prove the theorem with the additional restriction that $v_0\notin\{j_1,\dots,j_{2k}\}$. If $j_1,\dots,j_{2k}$ are all distinct, then \eqref{eq:umon} is a consequence of Lemma \ref{lem:u_2krcr}, \eqref{eq:Rmdef} and Proposition \ref{prop:Rmsgn}. To summarize, we just proved that under the assumptions of the theorem plus $v_0\notin\{j_1,\dots,j_{2k}\}$, we have
	\begin{equation}\label{eq:mond}
		(-1)^{k-1}\frac{\partial u_{2k}(\sigma_{j_1},\dots,\sigma_{j_{2k}})}{\partial J_{u_0v_0}}\geq 0 \text{ if }j_1,\dots,j_{2k} \text{ are all distinct.}
	\end{equation}
	As mentioned in Remark \ref{rem:gen}, since $u_{2k}=0$ if $\bf{J}=\bf{0}$, \eqref{eq:mond} implies
	\begin{equation}\label{eq:shld}
		(-1)^{k-1}\partial u_{2k}(\sigma_{j_1},\dots,\sigma_{j_{2k}})\geq 0 \text{ whenever }\bf{J}\geq\bf{0}, \text{ and }j_1,\dots,j_{2k} \text{ are all distinct.}
	\end{equation}
To prove the theorem, we still need to prove the following claim.
\begin{claim}\label{clm:monnd}
	Under the assumptions of the theorem plus $v_0\notin\{j_1,\dots,j_{2k}\}$, for each $k\in\mathbb{N}$, we have
	\begin{equation}\label{eq:monnd}
		(-1)^{k-1}\frac{\partial u_{2k}(\sigma_{j_1},\dots,\sigma_{j_{2k}})}{\partial J_{u_0v_0}}\geq 0 \text{ if }j_1,\dots,j_{2k} \text{ are not all distinct.}
	\end{equation}
\end{claim}
\begin{proof}
	If $k=1$, then $j_1=j_2$ and $u_2(j_1,j_2)=1$ by the definition \eqref{eq:ursdef1} or \eqref{eq:ursdef2} and spin-flip symmetry; the claim is trivial in that case. We may now assume that $k\geq 2$, and we prove the claim by induction on 
	\begin{equation}
		N:=\text{the number of pairs in } \{(j_p,j_q):j_p=j_q,1\leq p< q\leq 2k\}.
	\end{equation}
If $N=1$, then without loss of generality, we may assume that $j_1=j_2$. Lemma \ref{lem:reduction} implies that
\begin{align}\label{eq:reduction1}
	&(-1)^{k-1}\frac{\partial u_{2k}(\sigma_{j_1},\dots,\sigma_{j_{2k}})}{\partial J_{u_0v_0}}=\sum_{\substack{A: 1\in A, 2\notin A\\A\subseteq\{1,\dots,2k\}}}\bigg\{\left[(-1)^{\frac{|A|}{2}-1}\frac{\partial u_{|A|}(\sigma_{j_A})}{\partial J_{u_0v_0}}\right]\times\nonumber\\
	&\quad\left[(-1)^{\frac{2k-|A|}{2}-1}u_{2k-|A|}(\sigma_{j_{A^c}})\right]+\left[(-1)^{\frac{|A|}{2}-1}u_{|A|}(\sigma_{j_{A}})\right]\left[(-1)^{\frac{2k-|A|}{2}-1}\frac{\partial u_{2k-|A|}(\sigma_{j_{A^c}})}{\partial J_{u_0v_0}}\right]\bigg\};
\end{align}
only those $A$'s with $|A|\in2\mathbb{N}$ contribute to the sum by spin-flip symmetry. Since all vertices in $j_A=\{j_q:q\in A\}$ are distinct and all vertices in $j_{A^c}$ are distinct, \eqref{eq:mond} and \eqref{eq:shld} imply that each term in the brackets is $\geq 0$. This completes the proof of the base case. 

Suppose that the claim holds for any $N\leq M$ with $M\in\mathbb{N}$. This induction hypothesis and the same argument leading to \eqref{eq:shld} imply that
	\begin{equation}\label{eq:shld1}
	(-1)^{k-1}\partial u_{2k}(\sigma_{j_1},\dots,\sigma_{j_{2k}})\geq 0 \text{ if }  N\leq M.
\end{equation}
We now consider the case that $N=M+1$. Without loss of generality, we may assume that $j_1=j_2$.  Then \eqref{eq:reduction1} still holds, and for those A in \eqref{eq:reduction1} we have
\begin{align}
	&\left| \{(j_p,j_q):j_p=j_q,1\leq p< q\leq 2k, p\in A, q\in A\}\right|\leq M,\nonumber\\
	&\left| \{(j_p,j_q):j_p=j_q,1\leq p< q\leq 2k, p\in A^c, q\in A^c\}\right|\leq M
\end{align}
because the pair $(j_1,j_2)$ is in neither of the sets. So \eqref{eq:reduction1}, the induction hypothesis and \eqref{eq:shld1} (or possibly \eqref{eq:mond} and \eqref{eq:shld} if all vertices in $A$ or $A^c$ are distinct) imply that 
\begin{equation}
	(-1)^{k-1}\frac{\partial u_{2k}(\sigma_{j_1},\dots,\sigma_{j_{2k}})}{\partial J_{u_0v_0}}\geq 0 \text{ if } N=M+1.
\end{equation}
This completes the proof of the claim.
\end{proof}
This also completes the proof of the theorem.
\end{proof}

\section*{Acknowledgements}
The research of the second author was partially supported by NSFC grants 11901394 and 12271284. The authors thank Qi Hou for a critical reading of an earlier draft and Senya Shlosman for useful communications. The authors also thank the reviewers for many useful  comments and suggestions.

\section*{Declarations}
\subsection*{Data availability\nopunct}
Data sharing is not applicable to this article as no datasets were generated or analysed during the current study.
\subsection*{Conflict of interest\nopunct}
The authors have no competing interests to declare that are relevant to the content of this article. 

\bibliographystyle{abbrv}
\bibliography{reference}

\end{document}